\documentclass[a4paper,final]{amsart}

\usepackage[english]{babel}
\usepackage{graphicx}
\usepackage{mathtools} 		
\usepackage{amsmath, amssymb, amsfonts, amsthm} 
\usepackage[colorlinks=true,urlcolor=blue,linkcolor=black]{hyperref} 
\usepackage{url}            
\usepackage[dvipsnames]{xcolor}
\usepackage{color}          
\usepackage[markup=nocolor]{changes}		
\usepackage{pdfpages}
\usepackage{enumerate}

\theoremstyle{plain}
\newtheorem{lem}{Lemma}[section]
\newtheorem{thm}[lem]{Theorem}

\newtheorem{prop}[lem]{Proposition}

\theoremstyle{definition}
\newtheorem{defn}[lem]{Definition}
\newtheorem{rem}[lem]{Remark}

\newtheorem{hyp}[lem]{Hypothesis}

\newcommand{\NN}{\mathbb N}
\newcommand{\ZZ}{\mathbb Z}

\newcommand{\RR}{\mathbb R}
\newcommand{\R}{\mathbb R}

\newcommand{\EE}{\mathbb E}
\newcommand{\sA}{\mathcal{A}}

\newcommand{\sE}{\mathcal{E}}
\newcommand{\sG}{\mathcal{G}}
\newcommand{\sB}{\mathcal{B}}
\newcommand{\oF}{\textup{\textbf{F}}}

\newcommand{\oJ}{\textup{\textbf{J}}}
\newcommand{\intd}[1]{\, \mathrm{d}#1}

\numberwithin{equation}{section}

\usepackage[notcite,notref]{showkeys} 
\mathtoolsset{showonlyrefs}

\def\nicefrac#1#2{%
    \raise.5ex\hbox{$#1$}%
    \kern-.15em/\kern-.05em%
    \lower.25ex\hbox{$#2$}}
	
\begin{document}

\title[A dynamic approach to heterogeneous elastic wires]{A dynamic approach \\ to heterogeneous elastic wires}


\author[A.~Dall'Acqua]{Anna Dall'Acqua}
\address[A.~Dall'Acqua]{Institute of Applied Analysis, Helmholtzstra\ss e 18, 89081 Ulm, Germany.}
\email{anna.dallacqua@uni-ulm.de}
\author[L.~Langer]{Leonie Langer}
\address[L.~Langer]{Institute of Applied Analysis, Helmholtzstra\ss e 18, 89081 Ulm, Germany.}
\email{leonie.langer@uni-ulm.de}
\author[F.~Rupp]{Fabian Rupp}
\address[F.~Rupp]{Faculty of Mathematics, University of Vienna, Oskar-Morgenstern-Platz 1, 1090 Vienna, Austria.}
\email{fabian.rupp@univie.ac.at}

\begin{abstract}
We consider closed planar curves with fixed length and arbitrary winding number whose elastic energy depends on an additional density variable and a spontaneous curvature. Working with the inclination angle, the associated $L^2$-gradient flow is a nonlocal quasilinear coupled parabolic system of second order.  
We show local well-posedness, global existence of solutions, and full convergence of the flow for initial data in a weak regularity class.
\end{abstract}

\maketitle



\section{Introduction and main results}

In this work we consider a sufficiently regular closed planar curve $\gamma$ with density $\rho$ describing a heterogeneous elastic wire. The density $\rho$ can be used to model 
a distribution of mass or temperature, for instance. We assume that the bending stiffness of the wire depends on the density and consider the energy
\begin{align}
\label{eq:sE}
\mathcal{E}_\mu(\gamma,\rho)=\frac{1}{2}\int_\gamma\left(\beta(\rho)(\kappa-c_0)^2+\mu\left(\partial_s \rho\right)^2\right)\intd s,
\end{align}
with $\kappa$ the curvature of $\gamma$, the arc-length element $\intd s= |\partial_x\gamma(x)| \intd x$ and $\partial_s = \frac{1}{|\partial_x \gamma|} \partial_x$. Here $\mu>0$ models the diffusivity of the density, the smooth positive function $\beta\colon \RR \to\RR$ is the density-modulated stiffness and $c_0 \in \RR$ gives the spontaneous curvature. This energy is invariant  under orientation-preserving reparametrizations. The first part of the energy can be seen as a Helfrich-type energy for curves with modulated stiffness. The second part, the Dirichlet energy of $\rho$, ensures control of the density.

Modulo isometries of $\R^2$, the planar curve $\gamma$ (parametrized by arc-length) together with its orientation is uniquely determined by its inclination angle.
If the curve has length $L>0$, the energy can be expressed in terms of an inclination angle $\theta\colon [0,L]\to\RR$ and the density $\rho\colon[0,L]\to\RR$ by
\begin{align}
\label{Eallg}
\sE_\mu (\theta,\rho)=\frac{1}{2}\int_0^L\left(\beta(\rho)(\partial_s\theta-c_0)^2+\mu\left(\partial_s\rho\right)^2\right)\intd s,
\end{align}
using that $\kappa=\partial_s \theta$.
More precisely, we have 
$
\mathcal{E}_\mu(\gamma, \rho)= \sE_\mu (\theta, \rho),    
$
where $\theta$ is an inclination angle for $\gamma$.
Note that in \eqref{Eallg}, the differential operator $\partial_s$ is just the ordinary deriva\-tive with respect to $s$, the arc-length parameter. Similarly, $\intd s$ just denotes integration with respect to the variable $s\in [0,L]$.
A given function $\theta\colon[0,L]\to\RR$ represents the inclination angle of a $C^1$-closed curve if and only if 
\begin{align}\label{eq:angletocurve}
    \int_0^L\cos \theta\intd s  =\int_0^L\sin \theta \intd s = 0,\qquad 
    \theta(L)-\theta(0)=2\pi\omega,
\end{align}
for some $\omega\in \ZZ$, the rotation index of the curve. Notice that allowing $\omega$ to be negative we keep track of the orientation.

In the case of zero spontaneous curvature and fixed rotation index equal to one, this energy has previously been considered in \cite{BJSS2020}, where the minimization problem under the constraints of fixed length and fixed total mass (of the density) is studied. Especially in the case $c_0\neq 0$ it makes sense to fix the length of the curve, as otherwise a circle with suitable radius and a constant density is the trivial minimizer.

The goal of this article is to introduce a dynamic approach to minimize the elastic energy $\sE_\mu $ by evolving $(\theta, \rho)$ by the associated $L^2$-gradient flow.
In order to guarantee that $\theta$ describes a closed curve and to preserve the total mass of $\rho$, the evolution equations include nonlocal Lagrange multipliers. This results in the initial boundary value problem
\begin{align}
\label{eq:flow equation}
\qquad\begin{cases}
\begin{tabular}{l l l} 
 \multicolumn{2}{l }{$\displaystyle\partial_t\theta=\partial_s\left[\beta(\rho)(\partial_s\theta-c_0)\right]+\lambda_{\theta1}\sin\theta-\lambda_{\theta2}\cos\theta$\hphantom{-----}}  
 & in $(0,T) \times [0,L]$, \\
 \multicolumn{2}{l }{$\displaystyle\partial_t\rho=\mu\partial_s^2\rho-\frac{1}{2}\beta'(\rho)(\partial_s\theta-c_0)^2-\lambda_\rho$}
 & in $(0,T) \times [0,L]$, \\
 $\theta(\cdot, L)-\theta(\cdot,0)=2\pi \omega$,\hphantom{-----}
 & $\rho(\cdot,L)=\rho(\cdot,0)$
 & on $[0,T)$,\\
 $\partial_s \theta(\cdot,L)=\partial_s\theta(\cdot,0)$,
 & $\partial_s \rho(\cdot,L)=\partial_s \rho(\cdot,0)$
 & on $[0,T)$,\\
 $\theta(0,\cdot)=\theta_0$,
 & $\rho(0,\cdot)=\rho_0$ 
 & on $[0,L]$,
\end{tabular} 
\end{cases}
\end{align}
where $\theta_0$ satisfies \eqref{eq:angletocurve}.
The Lagrange multipliers $\lambda_{\theta 1}$, $\lambda_{\theta 2}$, $\lambda_\rho$ are chosen such that
\begin{align}\label{eq:constraint conservation}
    \qquad\int_0^L \cos \theta(t,s) \intd s = 0, \quad \int_0^L \sin \theta(t,s) \intd s = 0, \quad \int_0^L \rho(t,s) \intd s = \int_0^L\rho_0(s)\intd s
\end{align}
for all $t\in (0,T)$ (see Section \ref{mathset} for the explicit formulas), so that the evolution indeed corresponds to an evolution of closed curves with fixed total mass. The first boundary condition in \eqref{eq:flow equation} ensures that the rotation index is constant along the flow. The other boundary conditions are imposed in order to achieve that the corresponding curve is $C^2$-closed and that the density is $C^1$-periodic. Observe that the evolution equations are coupled so that we cannot treat them separately.

An essential feature of working with the angle function $\theta$ instead of the curve $\gamma$ is that the corresponding $L^2$-gradient flow of the energy \eqref{Eallg} is a parabolic system of second order in $(\theta, \rho)$.
In contrast, the Euler--Lagrange equation of \eqref{eq:sE} as a functional of $\gamma$ and $\rho$ is of fourth order in $\gamma$, like the classical elastica equation. Such an approach using the angle function has been used for instance in \cite{W1993} to prove well-posedness and global existence for a flow towards elastica.
Local and global existence of solutions for the classical elastic flow, given by a parabolic fourth order equation in $\R^n$, has been shown in several works, for instance \cite{Koiso,PoldenPDE,DKS,DAP2012,NovagaOkabe14,Wheeler15,DALP17,SpenerSTE17,PozziStinner,RS}. For a more detailed overview, see the recent survey \cite{MPP_Survey}.  A parabolic system related to \eqref{eq:flow equation} is studied in \cite{ABG22} and numerically elaborated in \cite{EGK2022}. 

In the following,  $\beta\in C^\infty(\R)$ satisfies $\beta(x)>0$ for all $x \in \RR$ and the model parameters
\begin{align*}
    L>0,\; \mu>0,\; \omega \in \ZZ \text{ and } c_0 \in \RR 
\end{align*}
are fixed. We summarize the compatibility conditions for the initial datum.
\begin{hyp}\label{hyp:intro}
Let $\theta_0, \rho_0\colon [0,L] \to \R$ be  $C^1$-functions such that 
\begin{align}\label{eq:bcid}
\begin{array}{lll}
    \theta_0( L)-\theta_0(0)=2\pi \omega, &  & \partial_s \theta_0(L)=\partial_s\theta_0(0),\\
     \rho_0(L)=\rho_0(0), &  & \partial_s \rho_0(L)=\partial_s \rho_0(0).
     \end{array}
\end{align}
Moreover, let $\theta_0$ satisfy \eqref{eq:angletocurve}.
\end{hyp}
In the following existence result, $h^{1+\alpha}$ is the little H\"older space, cf.\ Appendix \ref{appfuncspace}.

\begin{thm}[Local well-posedness]
\label{thm:Main STE}
Let $(\theta_0, \rho_0)\in h^{1+\alpha}([0,L])$ for some $\alpha\in (0,1)$ and assume Hypothesis \ref{hyp:intro}. Then there exists $T_0>0$ and a unique solution $(\theta, \rho)\in C^\infty((0,T_0)\times[0,L])$ of \eqref{eq:flow equation} on $(0,T_0)\times [0,L]$ which satisfies
\begin{align}\label{eq:loc wellposedness initial}
    \lim_{t\to 0} (\theta(t), \rho(t)) = (\theta_0, \rho_0) \quad \text{in } C^{1+\alpha}([0,L]).
\end{align}
Moreover, the solution depends continuously on the initial datum.
\end{thm}
In the literature, short-time existence and uniqueness for quasilinear parabolic systems is widely accepted if the initial datum is smooth. Our contribution provides a complete proof of local well-posedness for the system \eqref{eq:flow equation}, including continuous dependence on the data, even for initial data with low regularity and despite the presence of nonlocal terms, in a fairly concise way.

The idea of the proof is to meet the boundary conditions by formulating the problem in a periodic setting. 
Then, following the ideas in \cite{BCF2010}, the Inverse Function Theorem between appropriate time-dependent little H\"older spaces yields existence and continuous dependence on the initial datum. We deduce the local invertibility of the Fréchet derivative from the maximal regularity theory developed in \cite{lC2011} and a perturbation argument.
The precise formulation of the continuous dependence is given in Proposition \ref{prop:existence} and Remark \ref{rem:contdep} below.
To show instantaneous smoothing, we rely on the theory of parabolic systems in 
\cite{LSU1988} and a bootstrap argument. Finally, exploiting the dissipative structure of the system, uniqueness follows from an energy argument.

For the classical elastic flow, one expects long-time existence of the solution. We prove such a result in this much more general model.
\begin{thm}[Global existence]
\label{longtimeex} The solution $(\theta,\rho)$ in Theorem \ref{thm:Main STE} exists for all times, i.e.\ $T=\infty$.
\end{thm}

The flow equations \eqref{eq:flow equation} are a parabolic system, so we cannot use arguments based on the maximum principle to show global existence.
Instead, we rely on energy estimates to prove that suitable Sobolev norms along the solution cannot blow up in finite time. To that end, we first show boundedness of the velocity $\partial_t(\theta, \rho)$ in order to deal with the nonlinear coupling which limits the direct application of interpolation inequalities. 
To allow for $\omega=0$,
a new argument is needed to show a priori boundedness of the nonlocal Lagrange multipliers, compared to the strategy used in \cite{W1993}. 

Once global existence is established, it is natural to study the asymptotic behavior of the flow. Although we do not have a classification of the stationary solutions to \eqref{eq:flow equation} at hand, we can prove full convergence.
\begin{thm}[Long-time behavior]\label{thm:convergence}
Let $\beta\colon \R\to\R$ be real analytic. Then, the global solution $(\theta,\rho)$ in Theorem \ref{thm:Main STE} converges in $C^{2+\tilde\alpha}([0,L])$ for all $\tilde\alpha\in (0,\frac12)$ to a stationary solution of \eqref{Eallg} as $t\to\infty$, i.e.\ to a solution of
\begin{align}\label{eq:stationary}
\begin{cases}
\begin{tabular}{l l l} 
 \multicolumn{2}{l }{$\displaystyle 0=\partial_s\left[\beta(\rho)(\partial_s\theta-c_0)\right]+\lambda_{\theta1}\sin\theta-\lambda_{\theta2}\cos\theta$\hphantom{-----}}  
 & in $[0,L]$, \\
 \multicolumn{2}{l }{$\displaystyle 0=\mu\partial_s^2\rho-\frac{1}{2}\beta'(\rho)(\partial_s\theta-c_0)^2-\lambda_\rho$}
 & in $[0,L]$, \\
 $\theta(L)-\theta(0)=2\pi \omega$,\hphantom{-----}
 & $\rho(L)=\rho(0)$,
 &\\
 $\partial_s \theta(L)=\partial_s\theta(0)$,
 & $\partial_s \rho(L)=\partial_s \rho(0)$
 &
\end{tabular} 
\end{cases}
\end{align}
for some $\lambda_{\theta1}, \lambda_{\theta2}, \lambda_\rho\in \R$.
\end{thm}
From the energy estimates used in the proof of Theorem \ref{longtimeex}, we first deduce subconvergence of the solution. To prove full convergence, we assume $\beta$ to be analytic, since we rely on a suitable version of the {\L}ojasiewicz--Simon gradient inequality \cite{ConstrLoja}, a powerful functional analytic tool for studying the asymptotic behavior of (geometric) PDEs, even in the constrained setting; see, for instance, \cite{Simon83,CFS09,DPS16} and \cite{RS,VPWF,IsoWF}. 

Geometrically, our results can be interpreted as follows. The solution $(\theta(t), \rho(t))_{t\geq 0}$ describes a family of closed curves $(\gamma(t))_{t\geq 0}$ with density $(\rho(t))_{t\geq 0}$ that evolves in time decreasing the energy \eqref{eq:sE} and converging to an equilibrium as $t\to\infty$.

This article is structured as follows. In Section \ref{mathset}, we briefly review some basic geometric concepts concerning curves and give the formulas for the Lagrange multipliers.
The goal of Section \ref{ste} is to prove Theorem \ref{thm:Main STE}. First, we prove a local well-posedness result in a periodic setting and then show that this periodic solution gives a solution to \eqref{eq:flow equation}. Sections \ref{lte} and \ref{sec:longtimebehavior} are devoted to prove Theorem \ref{longtimeex} and Theorem \ref{thm:convergence}, respectively. 

\section{Preliminaries}\label{mathset}

We consider an arc-length parametrization of a closed regular curve $\gamma\colon [0,L]\to\R^2$ of length $L>0$ in the plane, which we assume to be sufficiently smooth. 
By a standard result in elementary differential geometry, there exists a  function $\theta\colon [0,L]\to\R$ such that $\partial_s \gamma =
(\cos\theta, \sin\theta)$.
Such an \emph{inclination angle} $\theta$ is not unique, in fact it is only unique up to addition of an integer multiple of $2\pi$. Nevertheless, many geometric quantities of the curve can be expressed in terms of an inclination angle; the (signed) \textit{curvature} $\kappa$ of the curve $\gamma$ is given by $\kappa=\partial_s \theta$ and its \emph{rotation index} $\omega\in \ZZ$ satisfies $2\pi \omega = \theta(L)-\theta(0)$.

In \eqref{eq:flow equation} three Lagrange multipliers appear. In the first equation, $\lambda_{\theta1}$ and $\lambda_{\theta2}$ ensure that $\gamma$ remains closed along the flow if the initial curve is closed. 
Indeed, defining 
\begin{align}
\label{eq:lambdatheta}
\begin{pmatrix}
\lambda_{\theta1}(t) \\ \lambda_{\theta2}(t)
\end{pmatrix}
:=\Pi^{-1}(\theta)(t)\int_0^L\begin{pmatrix}
-\sin\theta \\ \cos\theta
\end{pmatrix}
\partial_s\big(\beta(\rho)(\partial_s\theta-c_0)\big)\intd s,
\end{align}
where $\Pi^{-1}(\theta)(t)$ denotes the inverse of the matrix
\begin{align}\label{eq:def Pi}
\Pi(\theta)(t):=\begin{pmatrix}
\int_0^L\sin^2\theta\intd s & -\int_0^L\cos\theta\sin\theta\intd s \\
-\int_0^L\cos\theta\sin\theta\intd s & \int_0^L\cos^2\theta\intd s
\end{pmatrix},
\end{align}
for a sufficiently smooth solution $(\theta,\rho)$ of \eqref{eq:flow equation} we have 
\begin{align}\label{eq:dt lambda theta}
&\partial_t\int_0^L\begin{pmatrix}
\cos\theta \\ \sin\theta
\end{pmatrix}
\intd s
=0.
\end{align}

\begin{rem}\label{rem:det=0 iff theta const}
As already observed in \cite[Lemma 2.1]{W1993} we have
    \begin{align*}
        \det \Pi(\theta)(t) = \frac{1}{2} \int_0^L \int_0^L \sin^2\psi(t,s, \tilde{s})\intd s \intd \tilde{s},
    \end{align*}
    where $\psi(t,s, \tilde{s}) := \theta(t,s)-\theta(t,\tilde{s})$. Hence, as long as $\theta(t,\cdot)$ is not constant, the matrix $\Pi(\theta)(t)$ is invertible. 
    In particular, the Lagrange multipliers in \eqref{eq:lambdatheta} are well-defined, as long as the angle function describes a closed curve.
\end{rem}
\begin{rem}
\label{rem:lambdthetaafterpI}
The boundary conditions in \eqref{eq:flow equation} make it possible to use integration by parts to write the Lagrange multipliers $\lambda_{\theta1}$ and $\lambda_{\theta2}$ in the form
\begin{align*}
\begin{pmatrix}
\lambda_{\theta1}(t) \\ \lambda_{\theta2}(t)
\end{pmatrix}
=\Pi^{-1}(\theta)(t)\int_0^L\begin{pmatrix}
\cos\theta \\ \sin\theta
\end{pmatrix}
\partial_s\theta\,\beta(\rho)(\partial_s\theta-c_0)\intd s.
\end{align*}
\end{rem}

The third Lagrange multiplier $\lambda_\rho$ is chosen so that the total mass remains
fixed during the flow. Defining
\begin{align}\label{eq:lambdarho}
\lambda_\rho(t):=-\frac{1}{2L}\int_0^L\beta'(\rho)(\partial_s\theta-c_0)^2\intd s,
\end{align}
a sufficiently smooth solution $(\theta,\rho)$ of \eqref{eq:flow equation} satisfies
\begin{align}\label{eq:dt lambda rho}
\partial_t\int_0^L\rho\intd s
&=0.
\end{align}

From the fact that $\lambda_{\theta1}$ and $\lambda_{\theta2}$ ensure that the curve remains closed during the flow, we can already deduce that the integral over the inclination angle is preserved.
\begin{lem}
\label{intincang}
Let $T>0$ and $(\theta,\rho)$ be a sufficiently smooth solution of \eqref{eq:flow equation} on $(0,T)\times[0,L]$. 
Then for all $t\in (0,T)$ we have
$\int_0^L\theta\intd s=\int_0^L \theta_0 \intd s$.
\end{lem}
\begin{proof}
Integration of the evolution equation for $\theta$ in \eqref{eq:flow equation} yields
\begin{align*}
\partial_t\int_0^L\theta\intd s=\;&\beta(\rho)\left(\partial_s\theta-c_0\right)\Big\vert_0^L+\lambda_{\theta1}\int_0^L\sin\theta\intd s-\lambda_{\theta2}\int_0^L\cos\theta\intd s.
\end{align*}
By \eqref{eq:constraint conservation}, the integrals vanish. The boundary conditions in \eqref{eq:flow equation}
yield the claim. 
\end{proof}
A direct computation reveals that we have an $L^2$-gradient flow structure, i.e.\ for a sufficiently smooth solution $(\theta,\rho)$ of \eqref{eq:flow equation}, we have
\begin{align}\label{eq:energy decay}
    \frac{\intd}{\intd t}\sE_\mu (\theta,\rho)
    =-\int_0^L\left(\partial_t\theta\right)^2\intd s-\int_0^L\left(\partial_t\rho\right)^2\intd s.
\end{align}

\section{Local well-posedness and regularity}
\label{ste}
To prove Theorem \ref{thm:Main STE}, we will consider a periodic setting. This is quite natural due to the geometric interpretation of the problem. Doing so, we get rid of the boundary conditions which enables us to apply an existence result in \cite{lC2011} for a suitable linearization of our problem. We first prove short-time existence and continuous dependence on the data as well as uniqueness and smoothness of the solution to the periodic problem. Then, by a suitable transformation, we can transfer the results to problem \eqref{eq:flow equation}.

\subsection{Transformation to a periodic setting}
\label{subsec:trafoper}
For the moment, assume that $(\theta, \rho)$ is a sufficiently smooth function in $(t,s)$.
We observe that the boundary conditions in \eqref{eq:flow equation} imply that $\partial_s\theta$ is a $C^0$-periodic function, i.e.\ it can be extended to a continuous $L$-periodic function on $\RR$. Similarly, $\rho$ is a $C^1$-periodic function. 
The angle function $\theta$ is itself not periodic if the rotation number is different from zero. We hence first perform a suitable transformation.
Define the function 
\begin{align}\label{eq:defphi}
\phi \colon [0,L] \to \RR, \quad \phi(s):= \frac{2 \pi \omega s}{L}
\end{align}
and consider the function $u(t,s):=\theta(t,s)-\phi(s)$. The boundary conditions in \eqref{eq:flow equation} 
imply
that $u$ and $\rho$ are $C^1$-periodic functions on $[0,L]$. We thus consider the $L$-periodic extensions of $u_0:=\theta_0-\phi$ and $\rho_0$ to all of $\RR$ 
and look for $L$-periodic solutions $(u,\rho)$ defined on $(0,T)\times \RR$ for some $T>0$.
Now, instead of Hypothesis \ref{hyp:intro}, it is sufficient to assume that
\begin{align}
\label{eq:hypu}
    \int_0^L\cos(u_0+\phi)\intd s =\int_0^L\sin(u_0+\phi)\intd s=0.
\end{align}
In terms of $(u, \rho)$ $L$-periodic, \eqref{eq:flow equation} can be equivalently written as
\begin{align}
\label{eq:flow eq u} 
\begin{cases}
\displaystyle\partial_tu=\partial_s\left[\beta(\rho)\left(\partial_su+\partial_s \phi-c_0\right)\right]+\Lambda_u(u,\rho) &\textup{in } (0,T)\times{\RR},\\
\displaystyle\partial_t\rho=\mu\partial_s^2\rho-\frac{1}{2}\beta'(\rho)\left(\partial_su+\partial_s\phi-c_0\right)^2+\Lambda_\rho(u,\rho)\quad &\textup{in } (0,T)\times{\RR},\\
u(0,\cdot)=u_0, \qquad
\rho(0,\cdot)=\rho_0 \qquad & \textup{on } \RR.
\end{cases}
\end{align}
Here, to simplify the notation, instead of working with $(\lambda_{\theta1},\lambda_{\theta2})$ (see Remark \ref{rem:lambdthetaafterpI}) we consider $\Lambda_u$ given by
\allowdisplaybreaks
\begin{align}\nonumber
\Lambda_u(u,\rho)(t,s) &:=\begin{pmatrix}
\lambda_{\theta1}(u+\phi, \rho)(t)\\ \lambda_{\theta 2}(u+\phi, \rho)(t)
\end{pmatrix} \cdot \begin{pmatrix}
\sin(u(t,s)+\phi(s))\\ -\cos(u(t,s)+\phi(s))
\end{pmatrix}\\ \label{eq:LambdaupI}
&= \Pi^{-1}\left(u+\phi\right)\!
\int_0^L\begin{pmatrix}
\cos\left(u+\phi\right) \\ \sin\left(u+\phi\right)
\end{pmatrix}
\left(\partial_su+\partial_s\phi\right)\beta(\rho)\left(\partial_su+\partial_s\phi-c_0\right)\!\intd s \nonumber\\
&\quad\cdot\begin{pmatrix}
\sin\left(u+\phi\right) \\ -\cos\left(u+\phi\right)
\end{pmatrix}
=:\Pi^{-1}(u+\phi)\, \mathcal{J}(u, \rho) \cdot\begin{pmatrix}
\sin(u+\phi) \\ -\cos(u+\phi)
    \end{pmatrix}
\end{align}
for $(t,s)\in (0,T)\times\RR$ and, instead of $\lambda_{\rho}$, $\Lambda_\rho$ given by
\begin{align*}
\Lambda_\rho(u, \rho)(t):= \lambda_\rho(u+\phi, \rho)(t)=\frac{1}{2L}\int_0^L\beta'(\rho)\left(\partial_su+\partial_s \phi-c_0\right)^2\intd s,
\end{align*}
for $t \in (0,T)$, see \eqref{eq:lambdarho}.

We now define our concept of a solution. To that end, we work with time-weighted little Hölder spaces (see Appendix \ref{appfuncspace} for a precise definition), which provide an appropriate framework for the local well-posedness of the problem.

\begin{defn}
\label{defsol} 
Let $T>0$, $\alpha\in(0,1)$ and  $\eta\in\left(\frac{1}{2},1\right)$ be such that $2\eta+\alpha\not\in\ZZ$. Moreover, let $\left(u_0,\rho_0\right)\in h^{2\eta+\alpha}_L(\RR;\RR^2)$ be such that $u_0$ satisfies \eqref{eq:hypu}. We call $(u,\rho)$ a \emph{solution of \eqref{eq:flow eq u} on $(0,T)\times\RR$ with initial datum $(u_0, \rho_0)$} if 
\begin{align*}
    (u,\rho)\in\textup{BUC}^1_{1-\eta}\left([0,T^\prime];h^\alpha_L(\RR;\RR^2)\right)\cap\textup{BUC}_{1-\eta}\left([0,T^\prime];h^{2+\alpha}_L(\RR;\RR^2)\right)
\end{align*}
for all $T^\prime< T$, if $(u,\rho)$ solves the flow equations in \eqref{eq:flow eq u} on $(0,T)\times [0,L]$ and if
\begin{align*}
    \lim_{t\to0}(u,\rho)(t,\cdot)=(u_0,\rho_0)\quad \text{in }C^{2\eta+\alpha}(\RR;\R^2).
\end{align*}
\end{defn}

The time weight (described by the parameter $\eta$) allows us to consider less regular initial data. 
For $T>0$, $\alpha\in(0,1)$ and $\eta\in\left(\frac{1}{2},1\right)$ such that $2\eta+\alpha\not\in\ZZ$, we set
\begin{align}\label{eq:defEE0}
    &\quad\EE_0([0,T]):=\textup{BUC}_{1-\eta}\left([0,T];h^\alpha_L(\RR;\RR^k)\right),\\
    \label{eq:defEE1}
   &\quad\EE_1([0,T]):=\textup{BUC}_{1-\eta}^1\left([0,T]; h^{\alpha}_L(\RR;\RR^k)\right)\cap \textup{BUC}_{1-\eta}\left([0,T];h^{2+\alpha}_L(\RR;\RR^k)\right)
\end{align}
for $k=1,2$. For $k=2$, $\EE_1$ is the space of solutions in the sense of 
Definition \ref{defsol}. For $(u,\rho) \in \EE_1$ the right hand side of the evolution equations in \eqref{eq:flow eq u} belongs to $\EE_0$, cf.\ Remark \ref{rem:appbuc2}. We define $\gamma\EE_1:= h_L^{2\eta+\alpha}(\R;\R^k)$, the space for initial data (see \eqref{eq:tracespace}). In the following, the value of $k=1,2$ will be clear from the context.

We introduce the open set
\begin{align}
    \mathbb{U}([0,T]):=\big\lbrace (u, \rho) \in \EE_1([0,T]) : \inf_{[0,T]}\det \Pi(u+\phi)>0\big\rbrace
\end{align}
and define the nonlinear differential operator 
\begin{align}
    &\oF\colon \mathbb{U}([0,T])\to\EE_0([0,T]),\quad 
    \oF(u,\rho)=\left(\oF_1(u,\rho),\oF_2(u,\rho)\right),\\
\label{F1F2}    &\oF_1(u,\rho):=\partial_s\left[\beta(\rho)\left(\partial_su+\partial_s\phi-c_0\right)\right]+\Lambda_u(u,\rho),\phantom{\frac{1}{2}}\\
    &\oF_2(u,\rho):=\mu\,\partial_s^2\rho-\frac{1}{2}\beta'(\rho)\left(\partial_su+\partial_s\phi-c_0\right)^2+\Lambda_\rho(u,\rho).
\end{align}
The evolution equations in \eqref{eq:flow eq u} can be written as
$\displaystyle \partial_t( u,\rho )=\oF(u,\rho)$ on $(0,T)\times\RR.$
\begin{rem}
The choice $\eta>\frac{1}{2}$ ensures that the initial datum has finite energy, but also turns out to be essential to control the nonlinearities in the equation (cf.\  Lemma \ref{est:v-v0}, Proposition \ref{prop:sollinprobt}) and to show uniqueness (cf.\ Lemmas \ref{lem:unique}, \ref{lem:ds2u in L2}).
\end{rem}

\subsection{Short-time existence and continuous dependence}
\label{sec:ste&dep}

To prove short-time existence and continuous dependence on the data, we essentially follow the ideas of \cite[Theorem 2]{BCF2010}. However, since we work with periodic little H\"older spaces, we have to justify some steps differently.

\begin{prop}[Existence and continuous dependence]
\label{prop:existence} 
Let $\alpha\in(0,1)$ and $\eta\in\left(\frac{1}{2},1\right)$. Let $(u_0,\rho_0)\in \gamma\EE_1$ and let $u_0$ satisfy \eqref{eq:hypu}.
Then there exist $T_0 \in (0,\infty)$ and $r>0$ such that for every $(\tilde{u}_0,\tilde\rho_0)\in\gamma\EE_1$ with $\left\Vert (u_0,\rho_0)-(\tilde{u}_0,\tilde\rho_0)\right\Vert_{C^{2\eta+\alpha}([0,L])}<r$ the problem
\begin{align}
\label{eq:flow eq tildeu}
\begin{cases}
\displaystyle\partial_t\tilde{u}=\partial_s\left[\beta(\tilde\rho)\left(\partial_s\tilde{u}+\partial_s \phi-c_0\right)\right]+\Lambda_u(\tilde{u},\tilde\rho) &\textup{in } (0,T_0)\times{\RR},\\
\displaystyle\partial_t\tilde\rho=\mu\partial_s^2\tilde\rho-\frac{1}{2}\beta'(\tilde\rho)\left(\partial_s\tilde{u}+\partial_s\phi-c_0\right)^2+\Lambda_\rho(\tilde{u},\tilde\rho) &\textup{in } (0,T_0)\times{\RR}, \\
\tilde{u}(0,\cdot)=\tilde{u}_0,\qquad 
\tilde\rho(0,\cdot)=\tilde\rho_0 & \textup{on } {\RR},
\end{cases}
\end{align}
admits a solution $(\tilde{u},\tilde\rho)\in\EE_1([0,T_0])$.
Moreover, there is a $C^1$-map which associates to the initial datum a solution of the problem.
In particular, there exists a
solution $(u,\rho) \in \EE_1([0,T_0])$ of \eqref{eq:flow eq u} on $(0,T_0)\times\RR$ with initial datum $(u_0, \rho_0)$. 

\end{prop}

\begin{proof}
We divide the proof into several steps. For the sake of readability, the details are partially moved to Appendix \ref{sec:auxstep1}. We will refer to this at the respective points.

\textit{Step 1:} \textit{Construction of a ``good''~function.} We show that there exist $\bar{T}>0$ and $(\bar{u},\bar{\rho})\in \EE_1([0,\bar{T}])$ with $(\bar{u},\bar{\rho})(0, \cdot)=(u_0,\rho_0)$, such that $(\bar{u}, \bar{\rho})\in \mathbb{U}([0,\bar{T}])$, i.e.
\begin{equation}\label{eq:welldefdet}
\inf_{[0,\bar{T}]} \det\left(\Pi(\bar{u} + \phi)\right) >0, 
\end{equation} 
and having some extra regularity. More precisely, $\partial_s(\bar{u},\bar{\rho})\in
    \tilde\EE_1([0,\bar{T}])$ where
\begin{align}\label{eq:E1tilde}
    \tilde\EE_1([0,\bar{T}]):=\textup{BUC}_{1-\tilde\eta}^1\left([0,\bar{T}]; h^{\alpha}_L(\RR)\right)\cap \textup{BUC}_{1-\tilde\eta}\left([0,\bar{T}];h^{2+\alpha}_L(\RR)\right)
\end{align} 
for $\tilde\eta:=\eta-\frac{1}{2}>0$.\\
First, we take the solution $(\bar{u},\bar{\rho})\in\EE_1([0,\hat{T}])$, for $\hat{T}>0$ arbitrary, of 
\begin{align}
\label{eq:heat eq}
\begin{split}
    \begin{cases}
    \partial_t(v,\sigma)-\partial_s^2(v,\sigma)=(0,0) &\textup{ in } (0,\hat{T})\times\RR,\\
    (v,\sigma)(0,\cdot)=(u_0,\rho_0) &\textup{ on } \RR.
    \end{cases}
\end{split}
\end{align}
Such a solution exists by \cite[Theorem 6.4]{lC2011}. As solution of the heat equation, we know that $(\bar{u},\bar{\rho})\in C^\infty((0,\hat{T})\times\RR)$.
Due to the fact that $\eta>\frac{1}{2}$, 
$\partial_s(u_0,\rho_0)\in h^{2\tilde\eta+\alpha}_L(\RR)=\gamma\tilde\EE_1$ and hence it is an admissible initial datum. We look at 
\begin{align}
\label{eq:evprob lin ds}
\begin{split}
    \begin{cases}
    \partial_t(v,\sigma)-\partial_s^2(v,\sigma)=(0,0) &\textup{ in } (0,\hat{T})\times\RR,\\
    (v,\sigma)(0,\cdot)=\partial_s(u_0,\rho_0) &\textup{ on } \RR.
    \end{cases}
\end{split}
\end{align}
Again, we use \cite[Theorem 6.4]{lC2011} and obtain a unique solution in $\tilde\EE_1([0,\hat{T}])$ that is smooth for positive times. Since this solution coincides with $\partial_s(\bar{u},\bar\rho)$ for $t=0$ and $\partial_s(\bar{u},\bar\rho)$ solves the evolution equation in \eqref{eq:evprob lin ds}, we conclude by uniqueness that $\partial_s(\bar{u},\bar\rho)$ is the solution of \eqref{eq:evprob lin ds}, so that $\partial_s (\bar{u}, \bar{\rho})\in \tilde\EE_1([0,\hat{T
}])$.

It remains to check \eqref{eq:welldefdet}. Since $u_0+\phi$ satisfies Hypothesis \ref{hyp:intro}, using Remark \ref{rem:det=0 iff theta const} and Lemma \ref{appbuc1}, by continuity there exists $\bar{T}\leq \hat{T}$ such that \eqref{eq:welldefdet} is satisfied.

\textit{Step 2:} We claim that there exists $T^\prime\in(0,\bar{T}]$ such that the nonlinear operator
\begin{align}\label{eq:Frechet}
    &\Phi\colon \mathbb{U}([0,T'])\to\EE_0([0,T^\prime])\times\gamma\EE_1, \\ \nonumber
    &\Phi(\tilde{u},\tilde\rho):=\Big(\partial_t(\tilde{u},\tilde\rho)-\oF(\tilde{u},\tilde\rho),\; \left(\tilde{u},\tilde\rho\right)(0,\cdot)\Big)
\end{align}
is well-defined and a local diffeomorphism at $(\bar{u},\bar{\rho})$ (restricted to $[0,T^\prime]$). Here $\textup{\textbf{F}}$ is the map defined in \eqref{F1F2}. 

It is sufficient to prove
the existence of $T^\prime\in(0,\bar{T}]$ such that the Fréchet derivative 
\begin{align}
\label{eq:Frechet2}
\begin{split}
    &\textup{D}\Phi(\bar{u},\bar{\rho})\colon\EE_1([0,T^\prime])\to\EE_0([0,T^\prime])\times\gamma\EE_1, \\
    &\textup{D}\Phi(\bar{u},\bar{\rho})(v,\sigma):=\Big(\partial_t(v,\sigma)-\textup{D}\oF(\bar{u},\bar{\rho})(v,\sigma), (v, \sigma)(0,\cdot)\Big),
\end{split}
\end{align}
is a linear isomorphism. It then follows, since $\Phi\in C^1$, by the Local Inverse Function Theorem that $\Phi$ is a local diffeomorphism at $(\bar{u},\bar\rho)$.
The proof that the map in \eqref{eq:Frechet2} is a linear isomorphism is given in Proposition \ref{prop:sollinprobt} (with $\tilde\oJ = \textup{D}\Phi(\bar{u},\bar{\rho})$). The idea is to use a maximal regularity result in time-weighted little Hölder spaces \cite{lC2011}, where the time weight allows for rough initial data. Since the results in \cite{lC2011} are applicable to linear systems with time-independent coefficients, we first consider an auxiliary linear problem, cf.\ Proposition \ref{sollinprob}. A perturbation argument then gives that the map in \eqref{eq:Frechet2} is an isomorphism for small times.

\textit{Step 3:}
By Step 2, there is a neighborhood $V\subset \mathbb{U}([0, T'])$ of $(\bar{u},\bar{\rho})$ and a neighborhood $W\subset\EE_0([0,T^\prime])\times\gamma\EE_1$ of $\Phi(\bar{u},\bar{\rho})=:\big((\bar{f}_1,\bar{f}_2),(u_0,\rho_0)\big)$ such that $\Phi:V\to W$ is a diffeomorphism.
Choose $r>0$ small enough such that $B_{r}(\bar{f}_1,\bar{f}_2)\times B_r(u_0,\rho_0)\subset W$. Moreover, choose $T^{\prime\prime}\in(0,T^\prime)$ small enough so that 
\begin{equation}\label{eq:auf1}
\left\Vert(\bar{f}_1,\bar{f}_2)\right\Vert_{\EE_0([0,T''])}<r.
\end{equation}
Here it is crucial that the $\EE_0$-norm becomes arbitrary small for small times (see the definition of $\textup{BUC}_{1-\eta}$ in Section \ref{sec:BUC-spaces}).

\textit{Step 4:} 
Define $(\tilde{f}_1,\tilde{f}_2)\in\EE_0([0,T^\prime])$ by
\begin{align*}
    (\tilde{f}_1,\tilde{f}_2):=
    \begin{cases}
    (0,0) 
    &\textup{ on } \left[0,\frac12 T^{\prime\prime}\right],\\
    \frac{2}{T^{\prime\prime}} (t-\frac12 T^{\prime\prime})(\bar{f}_1,\bar{f}_2) 
    &\textup{ on } \left(\frac12 T^{\prime\prime},T^{\prime\prime}\right),\\
    (\bar{f}_1,\bar{f}_2) 
    &\textup{ on } \left[T^{\prime\prime},T^\prime\right].
    \end{cases}
\end{align*}
Then, by a short computation and the choice of $T''$ in \eqref{eq:auf1}, we have
\begin{align*}
    &\left\Vert (\tilde{f}_1,\tilde{f}_2)-(\bar{f}_1,\bar{f}_2)\right\Vert_{\EE_0([0,T^\prime])}
    <r.
\end{align*}
This implies that, with $(\tilde{u}_0,\tilde\rho_0)\in\gamma\EE_1$ as in the statement, we have
\begin{align*}
    \big((\tilde{f}_1,\tilde{f}_2),(\tilde{u}_0,\tilde\rho_0)\big)\in B_{r}(\bar{f}_1,\bar{f}_2)\times B_r(u_0,\rho_0)\subset W.
\end{align*}
Since $\Phi:V\to W$ is a diffeomorphism, there exists $(\tilde{u},\tilde{\rho})\in V\subset \mathbb{U}([0,T'])$ such that $\Phi(\tilde{u},\tilde\rho)=\big((\tilde{f}_1,\tilde{f}_2),(\tilde{u}_0,\tilde\rho_0)\big)$.
We restrict $(\tilde{u},\tilde{\rho})$ to $\big[0,\frac12 T^{\prime\prime}\big]\times\RR$, set $T_0:=\frac12 T^{\prime\prime}$ and obtain a solution of \eqref{eq:flow eq tildeu} 
on $(0,T_0)\times\RR$. 

\textit{Step 5:} 
We constructed a mapping which maps every $(\tilde{u}_0,\tilde\rho_0)\in B_r(u_0,\rho_0)$ to a solution $(\tilde{u},\tilde\rho)\in\EE_1([0,T_0])$ of the problem. It consists essentially of the composition of the $C^1$-mapping $\Phi^{-1}$ and a restriction operator and is thus continuously differentiable.
\end{proof}

\begin{rem}
\label{rem:contdep}
In the following section we prove uniqueness of solutions. Together with Proposition \ref{prop:existence} this yields the continuity --- or more precisely even the continuous differentiability --- of the mapping which maps any initial datum $({u}_0,\rho_0)\in\gamma\EE_1$ to the unique solution $({u},\rho)\in\EE_1([0,T_0])$ of \eqref{eq:flow eq u}.
\end{rem}

\subsection{Smoothness and uniqueness of solutions}
\label{sec:unique&smooth}

We now show that a solution is smooth for positive times. 

\begin{lem}[Smoothing]
\label{smoothing}
Let $0<T\leq\infty$ and $(u,\rho)$ be a solution, in the sense of Definition \ref{defsol}, 
 of \eqref{eq:flow eq u} 
on $(0,T)\times\RR$. 
Then $(u,\rho)\in C^\infty((0,T)\times\RR).$
\end{lem}
The idea is to define a smooth cut-off function, which vanishes at $t=0$ and outside of a spatial interval of length larger than $L$. The product of $u$ with this cut-off function is the solution of a new initial value problem --- this time with zero boundary values and zero initial data. This enables us to use \cite[Theorem IV.5.2]{LSU1988}, which gives us higher regularity of $u$ in the scale of parabolic H\"older spaces, cf.\ Appendix \ref{app:parhoelder}. We proceed similarly for $\rho$ and iterate this procedure to conclude smoothness.

\begin{proof}[Proof of Lemma \ref{smoothing}]
\!By Definition \ref{defsol} the solution $(u,\rho)$ satisfies $(u,\rho)\!\in\EE_1([0,T^\prime])$ for all $T^\prime<T$.
Let $0<\delta<T^\prime<T$ and let $\psi_1:[0,T]\times\RR\to\RR$ be a smooth cut-off function satisfying
\begin{align*}
    \psi_1(t,s)&\equiv 1 \textup{  for  } (t,s)\in\left[ \nicefrac{\delta}{2},T^\prime\right]
    \times\left[-\nicefrac{L}{4},\nicefrac{5L}{4}\right],\\
    \psi_1(t,s)&\equiv 0 \textup{  for  } (t,s)\in \left(\left[0,\nicefrac{\delta}{4}\right]\times\RR\right)\cup\left(\left[0,T^\prime\right]\times \left(\RR\setminus\left[-\nicefrac{L}{2},\nicefrac{3L}{2}\right]\right)\right).
\end{align*}
We define $u_1:=\psi_1 u\in\EE_1([0,T^\prime])$. By \eqref{eq:flow eq u} we have 
 \begin{align*}
    \partial_tu_1-a_{11}\partial_s^2u_1+a_1\partial_su_1=f_1,
 \end{align*}
where $a_{11}(t,s):=\beta(\rho)$, $a_1(t,s):=-\beta^\prime(\rho)\partial_s\rho,$ and
\begin{align*}
    f_1(t,s):=&-2\beta(\rho)\partial_su\partial_s\psi_1-\beta(\rho)u\partial_s^2\psi_1-\beta^\prime(\rho)u\partial_s\rho \partial_s\psi_1+\psi_1\beta^\prime(\rho)\partial_s\rho\left(\partial_s\phi-c_0\right)\\
    &\;+\psi_1\Lambda_u(u,\rho)+u\partial_t\psi_1.
\end{align*}
Defining $\zeta:=\eta+\frac{\alpha}{2}-\frac{1}{2}\in(0,1)$, Lemma \ref{E1BUC} in the appendix yields
\allowdisplaybreaks\begin{align*}
    \partial_su,\partial_s\rho\in H^{\frac{\zeta}{2},\zeta}\left([0,T^\prime]\times[-L,2L]\right).
\end{align*}
Since $\beta\in C^\infty(\RR)$, we conclude 
$a_{11},a_1,f_1,\in H^{\frac{\zeta}{2},\zeta}\left([0,T^\prime]\times[-L,2L]\right)$.
By \cite[Theorem IV.5.2]{LSU1988} there exists a unique solution $w \in H^{\frac{\zeta+2}{2},\zeta+2}\left([0,T^\prime]\times[-L,2L]\right)$ of
\begin{align}\label{eq:system Solonnikov}
    \begin{cases}
    \partial_tw-a_{11}\partial_s^2w+a_1\partial_sw=f_1\quad &\textup{ in } (0,T^\prime)\times(-L,2L),\\
    w=0 \quad& \textup{ on } (0,T^\prime)\times\left\lbrace-L,2L\right\rbrace,\\
    w(0,\cdot)=0 \quad& \textup{ on } (-L,2L).
    \end{cases}
\end{align}
By uniqueness $u_1=w\in H^{\frac{\zeta+2}{2},\zeta+2}\left([0,T^\prime]\times[-L,2L]\right)$. 
Analogously for $\rho$, we obtain
\begin{align*}
    u,\rho\in H^{\frac{\zeta+2}{2},\zeta+2}\left(\left[\nicefrac{\delta}{2},T^\prime\right]\times\left[-\nicefrac{L}{4},\nicefrac{5L}{4}\right]\right).
\end{align*}    

Since $u$ and $\rho$ are periodic with period length $L$ we even have
\begin{align}\label{eq:smoothing1}
    u,\rho\in H^{\frac{\zeta+2}{2},\zeta+2}\left(\left[\nicefrac{\delta}{2},T^\prime\right]\times\left[-L,2L\right]\right).
\end{align}  
Let $\psi_2:[0,T]\times\RR\to\RR$ be another smooth cut-off function satisfying
\begin{align*}
    \psi_2(t,s)&\equiv 1 \textup{  for  } (t,s)\in\left[\nicefrac{3\delta}{4},T^\prime\right]
    \times\left[-\nicefrac{L}{4},\nicefrac{5L}{4}\right],\\
    \psi_2(t,s)&\equiv 0 \textup{  for  } (t,s)\in \left(\left[0,\nicefrac{2\delta}{3}\right]\times\RR\right)\cup\left(\left[0,T^\prime\right]\times\RR\setminus\left[-\nicefrac{L}{2},\nicefrac{3L}{2}\right]\right).
\end{align*}
We set 
$
u_2:=\psi_2u 
\in H^{\frac{\zeta+2}{2},\zeta+2}\left(\left[\nicefrac{\delta}{2},T^\prime\right]\times[-L,2L]\right)\!.
$
Now, $u_2$ solves an initial boundary value problem similar to \eqref{eq:system Solonnikov} where $f_2$ is given by $f_1$ with $\psi_1$ replaced by $\psi_2$ and where the initial condition is prescribed at $t=\nicefrac{\delta}{2}$. By \eqref{eq:smoothing1}, we find that $a_{11},a_1,f_2\in H^{\frac{\zeta+1}{2},\zeta+1}\left(\left[\nicefrac{\delta}{2},T^\prime\right]\times[-L,2L]\right)$, so by \cite[Theorem IV.5.2]{LSU1988}, we conclude
$
u_2\in H^{\frac{\zeta+3}{2},\zeta+3}\left(\left[\nicefrac{\delta}{2},T^\prime\right]\times[-L,2L]\right)\!.
$
Proceeding similarly for $\rho$, we find
\begin{align*}
    u,\rho\in H^{\frac{\zeta+3}{2},\zeta+3}\left(\left[\nicefrac{3\delta}{4},T^\prime\right]\times[-L,2L]\right).
\end{align*}
Iterating the above procedure, by periodicity and since $0<\delta<T'<T$ were arbitrary, it follows $(u, \rho)\in C^\infty((0,T)\times\RR)$.
\end{proof}

We are now able to prove that solutions are unique. 

\begin{lem}[Uniqueness]
\label{lem:unique}
Let $(u_0,\rho_0)\in\gamma\EE_1$.
Let $0<T_1,T_2\leq\infty$
and let $(u_1,\rho_1)$ and $(u_2,\rho_2)$ be {solutions of \eqref{eq:flow eq u} 
on $(0,T_1)\times\RR$ and $(0,T_2)\times\RR$ with initial datum $(u_0, \rho_0)$}. Then $(u_1,\rho_1)\equiv(u_2,\rho_2)$ on $\left[0,\min\lbrace T_1,T_2\rbrace\right)$.
\end{lem}

\begin{proof}
Without loss of generality we assume $0<T_1\leq T_2$. Let $0<T<T_1$.
We show that the difference
$(u_d,\rho_d):=(u_1,\rho_1)-(u_2,\rho_2)\vert_{[0,T]}$
of the two solutions is zero 
using Gronwall's lemma. By Lemma \ref{smoothing} $(u_d, \rho_d)\in C^\infty((0,T]\times \RR)$, so
\begin{align*}
    \varphi(t):=\frac{1}{2}\int_0^L\left(\left(\partial_su_d\right)^2+\left(\partial_s\rho_d\right)^2+u_d^2+\rho_d^2\right)\intd s, \quad t\in (0,T],
\end{align*}
is differentiable.
In the following, $C$ denotes a constant, only depending on $u_1$, $u_2$, $\rho_1$, $\rho_2$, $T$, and the model parameters $\beta, c_0, L, \mu, \omega$, which may change from line to line. Other dependencies are given explicitly. Lemma \ref{appbuc1} and Remark \ref{rem:appbuc1} yield
\allowdisplaybreaks\begin{align}\label{eq:bounds uniqueness}
    \qquad\sup_{t\in[0,T]} \left\Vert (u_i, \rho_i)\right\Vert_{C^1}\leq C,\quad \sup_{[0,T]\times\RR} |\beta^{(k)}(\rho_i)| \leq C,\quad  \sup_{[0,T]} \left\Vert \Pi^{-1}(u_i+\phi)\right\Vert\leq C
\end{align}
for $i=1,2$, $k=0,1,2$.
In particular, there exists a compact interval $J\subset \R$ such that $\rho_i(t,s)\in J$ for all $(t,s)\in [0,T]\times \RR$ and $i=1,2$. Consequently, we have 
\allowdisplaybreaks\begin{align}\label{eq:beta bound below}
\textstyle
    \inf_{[0,T]\times \RR} \beta(\rho_i)\geq c>0\quad \text{for }i=1,2.
\end{align}
After integration by parts and using $\int_0^L\rho_d\intd s=0$ by \eqref{eq:dt lambda rho}, we get
\begin{align}
\label{I1+I2+J2}
\allowdisplaybreaks
    \frac{\intd}{\intd t}\varphi(t)
    &= -\int_0^L \beta(\rho_1)(\partial_s^2u_d)^2\intd s - \mu \int_0^L (\partial_s^2 \rho_d)^2 \intd s- \mu\int_0^L (\partial_s\rho_d)^2\intd s\nonumber\\ 
    & - \int_0^L \partial_s^2u_d\left(\beta(\rho_1)-\beta(\rho_2)\right)\partial_s^2u_2\intd s\nonumber
    \\
    &-\int_0^L\partial_s^2u_d\,\Big(\beta^\prime(\rho_1)\partial_s\rho_1\left(\partial_su_1+\partial_s\phi-c_0\right)+\Lambda_u(u_1,\rho_1)  \nonumber\\
    &\qquad\qquad\quad-\beta^\prime(\rho_2)\partial_s\rho_2\left(\partial_su_2+\partial_s\phi-c_0\right)-\Lambda_u(u_2,\rho_2)\Big)\intd s \nonumber\\
    &+\int_0^L\partial_s^2\rho_d\,\Big(\frac{1}{2}\beta^\prime(\rho_1)\left(\partial_su_1+\partial_s\phi-c_0\right)^2  -\frac{1}{2}\beta^\prime(\rho_2)\left(\partial_su_2+\partial_s\phi-c_0\right)^2\Big)\intd s \nonumber\\
    &+\int_0^Lu_d\,\Big(\beta(\rho_1)\partial_s^2u_1+\beta^\prime(\rho_1)\partial_s\rho_1\left(\partial_su_1+\partial_s\phi-c_0\right)+\Lambda_u(u_1,\rho_1)  \nonumber\\
    &\qquad\qquad\quad-\beta(\rho_2)\partial_s^2u_2-\beta^\prime(\rho_2)\partial_s\rho_2\left(\partial_su_2+\partial_s\phi-c_0\right)-\Lambda_u(u_2,\rho_2)\Big)\intd s \nonumber\\
    &+\int_0^L\rho_d\,\Big(-\frac{1}{2}\beta^\prime(\rho_1)\left(\partial_su_1+\partial_s\phi-c_0\right)^2 +\frac{1}{2}\beta^\prime(\rho_2)\left(\partial_su_2+\partial_s\phi-c_0\right)^2\Big)\intd s.
\end{align}
We estimate everything but the first three terms 
using Young's inequality to absorb higher order derivatives and bound the remaining ones by $\varphi(t)$ using \eqref{eq:bounds uniqueness}.
Indeed, since $\int_0^L\rho_d\intd s=0$, for each $t$ there is $s_t\in [0,L]$ with $\rho_d(t,s_t)=0$ by the Mean Value Theorem for Integrals. The Fundamental Theorem of Calculus thus yields
\begin{align}\label{eq:trick FTC}
    |\rho_d(t,s)|\leq \int_0^L |\partial_s\rho_d|\intd s.
\end{align}
Hence, the term on the second line in \eqref{I1+I2+J2} can be estimated by a constant times
\begin{align}
    \sup_{\RR}|\rho_d| \int_0^L |\partial_s^2u_d| |\partial_s^2u_2|\intd s \leq \delta\int_0^L (\partial_s^2 u_d)^2\intd s + C(\delta)\int_0^L (\partial_s^2u_2)^2\intd s \int_0^L(\partial_s\rho_d)^2 \intd s.
\end{align}
In the third line of \eqref{I1+I2+J2}, we add, subtract, and use Young's inequality to find
\allowdisplaybreaks\begin{align}
    &-\int_0^L\partial_s^2u_d\left(\beta^\prime(\rho_1)\partial_s\rho_1\left(\partial_su_1+\partial_s\phi-c_0\right)-\beta^\prime(\rho_2)\partial_s\rho_2\left(\partial_su_2+\partial_s\phi-c_0\right)\right)\!\intd s\\
    =&-\int_0^L\partial_s^2u_d\beta^\prime(\rho_1)\partial_s\rho_d\left(\partial_su_1+\partial_s\phi-c_0\right)\intd s-\int_0^L\partial_s^2u_d\beta^\prime(\rho_1)\partial_s\rho_2\partial_su_d\intd s\\
    &- \int_0^L\partial_s^2u_d\left(\beta^\prime(\rho_1)-\beta^\prime(\rho_2)\right)\partial_s\rho_2\left(\partial_su_2+\partial_s\phi-c_0\right)\intd s\\
&\qquad\leq
    \delta\int_0^L\left(\partial_s^2u_d\right)^2\intd s
    +C(\delta)\bigg(\int_0^L\left(\partial_s\rho_d\right)^2+\left(\partial_su_d\right)^2+\rho_d^2\intd s\bigg). \label{eq:gelato1}
\end{align}
On the other hand, for the term in the third line of \eqref{I1+I2+J2} involving $\Lambda_u$ we have
\begin{align}
    &-\int_0^L\partial_s^2u_d\left(\Lambda_u(u_1,\rho_1)-\Lambda_u(u_2,\rho_2)\right)\intd s\nonumber\\
    =&-\int_0^L\partial_s^2u_d\,
        \bigg[
        \Pi^{-1}\left(u_1+\phi\right)\,\mathcal{J}(u_1, \rho_1)\cdot\begin{pmatrix}
        \sin\left(u_1+\phi\right) \\ -\cos\left(u_1+\phi\right)
        \end{pmatrix}\nonumber\\
        &\qquad\qquad\quad -\Pi^{-1}\left(u_2+\phi\right)\,\mathcal{J}(u_2, \rho_2)\cdot\begin{pmatrix}
        \sin\left(u_2+\phi\right) \\ -\cos\left(u_2+\phi\right)
        \end{pmatrix}
        \bigg]
        \intd s.\label{eq:gelato0}
\end{align}    
We write the term in square brackets as
\allowdisplaybreaks\begin{align*}    
    &\left[
        \Pi^{-1}\left(u_1+\phi\right)-\Pi^{-1}\left(u_2+\phi\right)\right]\,\mathcal{J}(u_1, \rho_1)
        \cdot\begin{pmatrix}
        \sin\left(u_1+\phi\right) \\ -\cos\left(u_1+\phi\right)
        \end{pmatrix}
        \\
    &\quad+
        \Pi^{-1}\left(u_2+\phi\right)\left(\mathcal{J}(u_1,\rho_1)-\mathcal{J}(u_2, \rho_2)\right)
        \cdot\begin{pmatrix}
        \sin\left(u_1+\phi\right) \\ -\cos\left(u_1+\phi\right)
        \end{pmatrix}
        \\  
    &\quad+
        \Pi^{-1}\left(u_2+\phi\right) \mathcal{J}(u_2, \rho_2)
        \cdot\begin{pmatrix}
        \sin\left(u_1+\phi\right)-\sin\left(u_2+\phi\right) \\ -\cos\left(u_1+\phi\right)+\cos\left(u_2+\phi\right)
        \end{pmatrix}.
\end{align*}

Using the explicit formula for the inverse matrix, it can be shown that
\begin{align}
    \left\Vert \Pi^{-1}(u_1+\phi)-\Pi^{-1}(u_2+\phi)\right\Vert \leq C \int_0^L |u_d|\intd s.
\end{align}
Proceeding as in \eqref{eq:gelato1}, after adding and subtracting, we have
\begin{align}
    |\mathcal{J}(u_1, \rho_1)-\mathcal{J}(u_2, \rho_2)|\leq C\int_0^L \left(|u_d| + |\partial_s u_d|+|\rho_d|\right)\intd s.
\end{align}
Consequently,  \eqref{eq:gelato0} may be estimated by
\begin{align}
    \delta\int_0^L (\partial_s^2u_d)^2\intd s + C(\delta)\int_0^L u_d^2\intd s + C(\delta)\int_0^L (\partial_s u_d)^2 + C(\delta)\int_0^L\rho_d^2\intd s.
\end{align}
For the remaining terms, we proceed similarly. Using \eqref{eq:beta bound below} and choosing $\delta>0$ sufficiently small, the first two terms in \eqref{I1+I2+J2} can be used to absorb all terms involving a factor of $\delta$. Altogether, we obtain the differential inequality
\begin{align}
    \frac{\intd}{\intd t} \varphi(t) \leq C\Big(1+\int_0^L(\partial_s^2u_2)^2\intd s\Big)\varphi(t)\quad \text{for }t\in (0,T).
\end{align}
Using Gronwall's lemma, we conclude that for all $t'\in [0, T]$ we have
\begin{align}
    \varphi(t')\leq C\varphi(0) \exp{\Big(\int_0^{t'} \Big(1+\int_0^L (\partial_s^2 u_2)^2\intd s\Big) \intd t\Big)}.
\end{align}
By Lemma \ref{lem:ds2u in L2}, the integral is finite and uniqueness follows from $\varphi(0)=0$.
\end{proof}

\subsection{Proof of Theorem \ref{thm:Main STE}}
\label{sec:transferSTE}

It remains to transfer the results of Section \ref{sec:ste&dep} and Section \ref{sec:unique&smooth} to the initial nonperiodic setting and to prove Theorem \ref{thm:Main STE}.

\begin{proof}[Proof of Theorem \ref{thm:Main STE}]
We define $(u_0,\rho_0)$ as the $L$-periodic extension of $(\theta_0-\phi,\rho_0)$ to all of $\RR$. Then $(u_0,\rho_0)\in h_L^{1+\alpha}(\RR)$. Now we choose $\bar{\alpha}\in(0,1)$ and $\bar{\eta}\in\left(\frac{1}{2},1\right)$ such that $2\bar{\eta}+\bar{\alpha}= 1+\alpha$. Since $(u_0,\rho_0)\in h^{2\bar{\eta}+\bar{\alpha}}_L(\RR)$, Proposition \ref{prop:existence} yields $\bar{T}>0$ and a solution 
\begin{align}\label{eq:STE BUC regularity}
    (u,\rho)\in\textup{BUC}_{1-\bar{\eta}}^1\left([0,\bar{T}]; h^{\bar{\alpha}}_L(\RR)\right)\cap \textup{BUC}_{1-\bar{\eta}}\left([0,\bar{T}];h^{2+\bar{\alpha}}_L(\RR)\right)
\end{align}
of \eqref{eq:flow eq u} on $(0,\bar{T})\times\RR$ that satisfies
$
    \lim_{t\to0}(u,\rho)(t)=(u_0,\rho_0) \text{ in } C^{2\bar{\eta}+\bar{\alpha}}(\RR)
$
by Lemma \ref{appbuc1}.
According to Proposition \ref{lem:unique}, the solution is unique and due to Lemma \ref{smoothing}  $(u,\rho)\in C^\infty((0,\bar{T})\times\RR)$. Now we define $(\theta,\rho)$ as the restriction of $(u+\phi,\rho)$ on $[0,\bar{T})\times [0,L]$. By construction $(\theta,\rho)$ solves \eqref{eq:flow equation} on $(0,\bar{T})\times[0,L]$. Moreover, the $C^1$-periodicity of $(u,\rho)$ ensures, that $(\theta,\rho)$ satisfies the boundary conditions in \eqref{eq:flow equation}.
Furthermore, the choice of $\bar{\alpha}$ and $\bar{\eta}$ as $2\bar{\eta}+\bar{\alpha}= 1+\alpha$ gives us $\lim_{t\to 0} (\theta(t), \rho(t)) = (\theta_0, \rho_0)$ in $C^{1+\alpha}([0,L];\R^2)$. The uniqueness of the solution, its smoothness and the continuous dependence on the initial datum transfer from $(u,\rho)$ to $(\theta,\rho)$. 
\end{proof}

\begin{rem}
\label{rem:ste+}
In the above proof of Theorem \ref{thm:Main STE} we chose $\bar{\alpha}\in(0,1)$ and $\bar{\eta}\in\left(\frac{1}{2},1\right)$ such that $2\bar{\eta}+\bar{\alpha}= 1+\alpha\in(1,2)$. This yields the convergence of the solution $(\theta,\rho)$ to the initial datum $(\theta_0,\rho_0)$ in $C^{1+\alpha}([0,L])$. But in fact, in Proposition \ref{prop:existence} we can allow for arbitrary $\bar{\alpha}\in(0,1)$ and $\bar{\eta}\in\left(\frac{1}{2},1\right)$ with just the condition that $2 \bar{\eta}+\bar{\alpha} \not \in \ZZ$. By exploiting this freedom we observe the following.

If we even have $(\theta_0,\rho_0)\in h^{2+\alpha}([0,L])$ for $\alpha\in(0,1)$ (with periodic second derivatives), we get convergence of the solution $(\theta,\rho)$ to the initial datum $(\theta_0,\rho_0)$ in $C^{2+\alpha}([0,L])$. This is an important ingredient in proving global existence, cf.\ Section \ref{sec:thm1.3} below.
\end{rem}
\section{Global existence}
\label{lte}
Let $(\theta_0, \rho_0)\in h^{1+\alpha}([0,L])$, $\alpha\in (0,1)$ be such that Hypothesis \ref{hyp:intro} is satisfied.
As a consequence of Theorem \ref{thm:Main STE}, a standard argument for initial value problems yields the existence of a \emph{maximal existence time} $T_{\textup{max}}\in (0,\infty]$ such that the unique solution $(\theta,\rho)$ to \eqref{eq:flow equation} with this initial datum can be smoothly extended to $(0, T_{\textup{max}})$.

The goal of this section is to prove Theorem \ref{longtimeex}, i.e.\ that $T_{\textup{max}}=\infty$.

\subsection{\texorpdfstring{Controlling the matrix $\Pi^{-1}$}{Controlling the matrix Pi-1}
}

As a first step, we show that the inverse of the matrix in \eqref{eq:def Pi}, which defines the Lagrange-multipliers in \eqref{eq:lambdatheta}, does not degenerate if the $L^2$-norm of the curvature $\kappa=\partial_s \theta$ is bounded. In contrast to \cite[Lemma 2.1]{W1993}, our proof is variational and does not depend on the winding number.
\begin{lem}
\label{lem:boundPiFabi}
    Let $K>0$. Then there exists $M=M(K,L)>0$ such that for all functions $\theta\in W^{1,2}(0,L)$ satisfying \eqref{eq:angletocurve} and
    \begin{align}\label{eq:bound elastic}
        \int_0^L(\partial_s \theta)^2\intd s\leq K
    \end{align}
    we have
    \begin{align*}
    \left\Vert\Pi^{-1}(\theta)\right\Vert
    :=\sup_{{X\in\RR^2, \left\vert X\right\vert=1}}\left\vert\Pi^{-1}(\theta\right)(X)\vert\leq M.
    \end{align*}
\end{lem}
\begin{proof}
    By \eqref{eq:def Pi}, all entries of the matrix $\Pi(\theta)$ are bounded by $L$, thus it suffices to bound $\det \Pi(\theta)$ from below for all $\theta\in W^{1,2}(0,L)$ satisfying \eqref{eq:angletocurve} and \eqref{eq:bound elastic} by some constant $m=m(K,L)>0$. Let $(\theta_k)_{k\in \NN} \subset W^{1,2}(0,L)$ be a minimizing sequence 
    for
    \begin{align}
        m:=\inf\left\lbrace\det \Pi(\theta) \mid \theta \in W^{1,2}(0,L) \text{ satisfying } \eqref{eq:angletocurve} \text{ and } \eqref{eq:bound elastic}\right\rbrace.
    \end{align}
    Note that $m\geq 0$ by Remark \ref{rem:det=0 iff theta const}. Without loss of generality, we may assume $\theta_k(0)=0$, after replacing $\theta_k$ by $\theta_k-\theta_k(0)$. Indeed, by direct computation \eqref{eq:bound elastic} and also \eqref{eq:angletocurve} are invariant under addition of a constant. Moreover, Remark \ref{rem:det=0 iff theta const} directly implies that $\det \Pi(\theta_k-\theta_k(0))=\det\Pi(\theta_k)$.
    
    Consequently, $(\theta_k)_{k\in \NN}\subset W^{1,2}(0,L)$ is bounded, so that we may assume $\theta_k \rightharpoonup \theta$ in $W^{1,2}(0,L)$ and $\theta_k \to \theta$ in $C([0,L])$ by Rellich--Kondrachov. 
    The limit $\theta$ clearly satisfies \eqref{eq:angletocurve} and, by weak lower semicontinuity of the $L^2$-norm, also \eqref{eq:bound elastic}. From \eqref{eq:def Pi}, we find $\det\Pi(\theta_k)\to\det \Pi(\theta)$ and thus conclude that $m = \det\Pi(\theta)$ is a minimum. Since $\theta$ is continuous, by Remark \ref{rem:det=0 iff theta const} we find that $m=0$ if and only if $\theta$ is constant. However, the latter is impossible due to \eqref{eq:angletocurve}, so $m>0$ and the claim follows.
\end{proof}

\subsection{Energy estimates}

We now prove norm estimates for our solution, which by Lemma \ref{smoothing} instantaneously becomes smooth. Since we are only interested in long-time existence here, for the rest of this subsection we assume that we have a fixed \emph{smooth} solution $(\theta,\rho)\in C^\infty([0,T)\times[0,L])$ of \eqref{eq:flow equation}. This means that the initial datum $(\theta_0, \rho_0)$ is smooth and satisfies Hypothesis \ref{hyp:intro}.

Furthermore, throughout this subsection $C \in (0,\infty)$ denotes a generic constant that may vary from line to line. The constant is only allowed to depend on the fixed model parameters 
as well as the initial datum $(\theta_0, \rho_0)$. In particular, $C$ does not depend on the (not necessarily maximal) existence time $T$. 
This is essential for proving the convergence in Section \ref{sec:convergence}. First, we gather some direct consequences of the constrained gradient flow structure.

\begin{prop}
\label{prop:ltest1} 
We have
    \begin{enumerate}[(i)]
        \item $\left\Vert\theta\right\Vert_{L^\infty\left(\left[0,T\right)\times[0,L]\right)}\leq C$,\quad
        $\left\Vert\rho\right\Vert_{L^\infty\left(\left[0,T\right)\times[0,L]\right)}\leq C$,
        \item $\sup_{t\in [0,T)}\left\Vert\partial_s\theta\right\Vert_{L^2(0,L)}\leq C$,\quad
        $\sup_{t\in [0,T)}\left\Vert\partial_s\rho\right\Vert_{L^2(0,L)}\leq C$,
        \item $\left\Vert\partial_t\theta\right\Vert_{L^2\left(\left[0,T\right)\times[0,L]\right)}\leq C$,\quad
        $\left\Vert\partial_t\rho\right\Vert_{L^2\left(\left[0,T\right)\times[0,L]\right)}\leq C$.
    \end{enumerate}
\end{prop}
\begin{proof}
Statement (iii) follows from integrating \eqref{eq:energy decay} in time and using $\sE_\mu (\theta,\rho)\geq 0$.
\newline
Moreover, for any $t\in [0,T)$ we obtain
\begin{align*}
    \sE_\mu (\theta_0, \rho_0)\geq \frac{\mu}{2} \int_0^L (\partial_s\rho)^2\intd s
\end{align*}
and the second estimate in (ii) is proven.\newline
We will now use this to get an $L^\infty$-bound on $\rho$. Since the total mass is conserved along the evolution by \eqref{eq:constraint conservation}, the Mean Value Theorem for integrals implies the existence of $s_t\in [0,L]$ with $\rho(t, s_t)=\frac1L\int_0^L\rho_0\intd s$ for any $t\in [0,T)$. Similar to \eqref{eq:trick FTC}, the Fundamental Theorem of Calculus yields
\begin{align}\label{eq:Linfty interpolation}
    \vert \rho(t,s)-\rho(t,s_t)\vert \leq \int_0^L \vert \partial_s \rho\vert\intd s \leq L^{\frac{1}{2}}\Big(\int_0^L (\partial_s\rho)^2\intd s\Big)^{\frac{1}{2}} \leq C
\end{align}
and the second estimate of (i) is proven. 
\newline 
Consequently, there exists a compact interval $J\subset \RR$ only depending on $\theta_0, \rho_0$ and the model parameters such that
\begin{align}\label{eq:Jbeta}
    \rho(t,s) \in J \quad\text{for all }t\in [0,T),\, s\in [0,L].
\end{align}
By continuity, we find that 
\begin{align}\label{eq:beta>0}
    \inf_{J}\beta\geq C>0.
\end{align}
 Therefore, using \eqref{eq:energy decay} again, we obtain
\begin{align*}
    \sE_\mu (\theta_0, \rho_0)\geq \inf_{J}\beta \int_0^L(\partial_s \theta-c_0)^2\intd s \geq C \Big(\int_0^L (\partial_s \theta)^2\intd s - 2c_0 \int_0^L\partial_s\theta\intd s + c_0^2 L\Big).
\end{align*}
By \eqref{eq:angletocurve}, we have $\int_0^L\partial_s \theta \intd s = 2\pi\omega$, and hence the first part of (ii) follows. \newline
This together with the fact that $\int_0^L\theta\intd s = \int_0^L\theta_0\intd s$ by Lemma \ref{intincang} yields the first estimate in (i) by proceeding in the same fashion as in \eqref{eq:Linfty interpolation}.
\end{proof}

\begin{prop}
\label{lmultbound}
For the Lagrange multipliers we have
\begin{align*}
    \sup_{t\in[0,T)}\left\vert\lambda_{\theta1}\right\vert\leq C,
    \quad \sup_{t\in[0,T)}\left\vert\lambda_{\theta2}\right\vert\leq C,
    \quad \sup_{t\in[0,T)}\left\vert\lambda_{\rho}\right\vert\leq C.
\end{align*}
\end{prop}

\begin{proof}
Because of Proposition \ref{prop:ltest1}, Lemma \ref{lem:boundPiFabi} can be applied with $K=C$, only depending on the model parameters and the initial datum.
Hence, for the Lagrange multipliers $\lambda_{\theta1}$ and $\lambda_{\theta2}$ by Remark \ref{rem:lambdthetaafterpI} we have 
\begin{align*}
    \left\vert\begin{pmatrix}
    \lambda_{\theta1}\\ \lambda_{\theta2}
    \end{pmatrix}\right\vert
    &\leq C\sup_{J}\beta\Big(\int_0^L(\partial_s\theta)^2\intd s+\int_0^L\left\vert c_0\,\partial_s\theta\right\vert\intd s\Big),
\end{align*}
where $J$ is as in \eqref{eq:Jbeta}. The right hand side is bounded for all $t\in [0,T)$ by Proposition \ref{prop:ltest1} and the Cauchy--Schwarz inequality. This way, we obtain the desired bounds for $\lambda_{\theta1}$ and $\lambda_{\theta2}$.
For $\lambda_\rho$ we use \eqref{eq:lambdarho}
and proceed similarly, using that $\sup_J |\beta'|\leq C$ by continuity.
\end{proof}

As a next step, we would like to bound the $L^2$-norm of $\partial_s^2(\theta, \rho)$ uniformly in time. Directly pursuing this idea, we encounter difficulties in controlling the nonlinear coupling of the system \eqref{eq:flow equation}, which seems to be incompatible with a direct application of interpolation inequalities. Instead, we control the $L^2$-norm of the velocity $\partial_t (\theta, \rho)$ first.

\begin{prop}
\label{dtLinfL2}
We have
\begin{align*}
    \sup_{t\in[0,T)}\left\Vert\partial_t\theta\right\Vert_{L^2(0,L)}\leq C,\qquad
    \sup_{t\in[0,T)}\left\Vert\partial_t\rho\right\Vert_{L^2(0,L)}\leq C.
\end{align*}
\end{prop}

\begin{proof}
By continuity, Proposition \ref{prop:ltest1} and with $J$ as in \eqref{eq:Jbeta} we have that
\begin{align}\label{eq:bound Dbeta D2beta}
    \sup_{J}\left\vert\beta^\prime\right\vert\leq C,\qquad
    \sup_{J}\left\vert\beta^{\prime\prime}\right\vert\leq C.
\end{align}
We consider the smooth function
\begin{align*}
    \varphi(t):=\frac{1}{2}\int_0^L\left((\partial_t\theta)^2+(\partial_t\rho)^2\right)\intd s, \quad \text{for }t\in [0,T).
\end{align*}
We now use \eqref{eq:flow equation} and integration by parts 
to differentiate $\varphi$. The boundary terms dissapear as a consequence of the boundary conditions in \eqref{eq:flow equation}.
We have
\begin{align}\label{dtvarphi2}
    \frac{\intd}{\intd t}\varphi(t)
    =\, &-\int_0^L\beta(\rho)(\partial_t\partial_s\theta)^2\intd s
    -\int_0^L\mu(\partial_t\partial_s\rho)^2\intd s\nonumber\\ 
    &-2\int_0^L\partial_t\partial_s\theta\,\beta^\prime(\rho)\,\partial_t\rho\,\left(\partial_s\theta-c_0\right) \intd s
    -\frac{1}{2}\int_0^L(\partial_t\rho)^2\beta^{\prime\prime}(\rho)\,(\partial_s\theta-c_0)^2\intd s\nonumber\\
    &+\int_0^L(\partial_t\theta)^2\lambda_{\theta1}\cos\theta\intd s
    +\int_0^L(\partial_t\theta)^2\lambda_{\theta2}\sin\theta\intd s.
\end{align}
By \eqref{eq:dt lambda theta} and \eqref{eq:dt lambda rho} no time derivatives of the Lagrange multipliers appear.
Using \eqref{eq:bound Dbeta D2beta}, Young's inequality, Proposition \ref{prop:ltest1} and Proposition \ref{lmultbound}, we have
\begin{align}\label{eq:dtvarphi3}
    \frac{\intd}{\intd t}\varphi(t)
    \leq\,&-\inf_{J}\beta\int_0^L(\partial_t\partial_s\theta)^2\intd s
    -\mu\int_0^L(\partial_t\partial_s\rho)^2\intd s +\delta\int_0^L(\partial_t\partial_s\theta)^2\intd s\nonumber\\ 
    &+C(\delta) \sup_{s\in[0,L]}(\partial_t\rho)^2
    +C\int_0^L(\partial_t\theta)^2\intd s
\end{align}
for $\delta>0$ to be chosen and some correspondingly adjusted $C(\delta)>0$.
Using \eqref{eq:dt lambda rho} and arguing as in \eqref{eq:Linfty interpolation}, we find there exists $s_t\in [0,L]$ with $\partial_t \rho(t, s_t)=0$ and thus
\begin{align}
\label{eq:GagNirdtrho}
    \sup_{s\in[0,L]}\left\vert\partial_t\rho\right\vert^2 &= \sup_{s\in [0,L]}\left\vert (\partial_t \rho)^2 - (\partial_t\rho(t,s_t))^2\right\vert
    \leq \int_0^L \left\vert \partial_s\left(\partial_t\rho\right)^2\right\vert\intd s\nonumber\\
    &   \leq 2\left\Vert\partial_t\rho\right\Vert_{L^2(0,L)}\left\Vert\partial_t\partial_s\rho\right\Vert_{L^2(0,L)}
    \leq \delta\left\Vert\partial_t\partial_s\rho\right\Vert_{L^2(0,L)}^2+ C(\delta)\left\Vert\partial_t\rho\right\Vert_{L^2(0,L)}^2,
\end{align}
using Young's inequality in the last step.
Using \eqref{eq:beta>0} and choosing $\delta>0$ small enough in \eqref{eq:dtvarphi3} and \eqref{eq:GagNirdtrho}, we obtain
\begin{align}\label{eq:gronwall linear}
    \frac{\intd}{\intd t}\varphi(t)\leq C\varphi(t)\quad \text{for all }t\in [0,T).
\end{align}
By Proposition \ref{prop:ltest1}, we find that $\int_0^T \varphi(t)\intd t$ is bounded independently of $T$ and thus the claim follows from integrating \eqref{eq:gronwall linear}.
\end{proof}

\begin{prop}
\label{ltestds2} 
We have
\begin{align*}
    \sup_{t\in[0,T)}\left\Vert\partial_s^2\theta\right\Vert_{L^2(0,L)}\leq C,\qquad
    \sup_{t\in[0,T)}\left\Vert\partial_s^2\rho\right\Vert_{L^2(0,L)}\leq C.
\end{align*}
\end{prop}

\begin{proof}
To prove the required estimates, we need to consider simultaneously the evolution equations of $\theta$ and $\rho$.
Recall from \eqref{eq:flow equation} that we have
\begin{align*}
    \beta(\rho)\partial_s^2\theta&=\partial_t\theta-\beta^\prime(\rho)\partial_s\rho(\partial_s\theta-c_0)-\lambda_{\theta1}\sin\theta+\lambda_{\theta2}\cos\theta,\\
    \mu\partial_s^2\rho&=\partial_t\rho+\frac{1}{2}\beta^\prime(\rho)(\partial_s\theta-c_0)^2+\lambda_\rho
\end{align*}
on $[0,T)$. Using \eqref{eq:beta>0} and arguing as in \eqref{eq:bound Dbeta D2beta}, by Young's inequality we find
\begin{align*}
    &\left\Vert\partial_s^2\theta\right\Vert^2_{L^2(0,L)}+\left\Vert\partial_s^2\rho\right\Vert^2_{L^2(0,L)}\\
    &\leq C\bigg(\int_0^L(\partial_t\theta)^2\intd s+\int_0^L(\partial_s\rho)^2(\partial_s\theta-c_0)^2\intd s+\lambda_{\theta1}^2\int_0^L\!\sin^2\theta\intd s+ \lambda_{\theta2}^2\int_0^L\!\cos^2\theta\intd s\\
    &\quad+\int_0^L(\partial_t\rho)^2\intd s+\int_0^L(\partial_s\theta-c_0)^4\intd s+\lambda_\rho^2L\bigg)
\end{align*}
for all $t\in[0,T)$.
By Propositions \ref{lmultbound} and \ref{dtLinfL2} it is sufficient to consider the second term and the sixth term on the right side of the inequality. Using Hölder's inequality, interpolation \cite[p.~233]{Brezis} and  $\int_0^L\partial_s \rho\intd s=0$, we find 
\begin{align*}
    &\left\Vert(\partial_s\rho)^2(\partial_s\theta)^2\right\Vert_{L^1(0,L)}
    \leq\left\Vert\partial_s\rho\right\Vert^2_{L^4(0,L)}\left\Vert\partial_s\theta\right\Vert^2_{L^4(0,L)}\\
    &\qquad  \leq C \left\Vert\partial_s^2\rho\right\Vert_{L^2(0,L)}^{\frac{1}{2}}\left\Vert\partial_s\rho\right\Vert_{L^2(0,L)}^{\frac{3}{2}}\Big(\left\Vert\partial_s^2\theta\right\Vert_{L^2(0,L)}^{\frac{1}{4}}\left\Vert\partial_s\theta\right\Vert_{L^2(0,L)}^{\frac{3}{4}}+\left\Vert\partial_s\theta\right\Vert_{L^2(0,L)}\Big)^2.
\end{align*}
Due to Proposition \ref{prop:ltest1} the norm $\left\Vert\partial_s\theta\right\Vert_{L^2(0,L)}$ is bounded for all $t\in[0,T)$ and similarly for $\rho$. Therefore, using twice Young's inequality we estimate
\begin{align}\label{eq:absorb1}
    \left\Vert(\partial_s\rho)^2(\partial_s\theta)^2\right\Vert_{L^1(0,L)}
    &\leq \delta\left\Vert\partial_s^2\rho\right\Vert_{L^2(0,L)}^2+\delta\left\Vert\partial_s^2\theta\right\Vert_{L^2(0,L)}^2+C(\delta),
\end{align}
for $\delta>0$ to be chosen. In an analogous way we also get
\begin{align}
\label{eq:absorb2}
    \left\Vert(\partial_s\theta)^4\right\Vert_{L^1(0,L)}
    \leq \delta\left\Vert\partial_s^2\theta\right\Vert_{L^2(0,L)}^2+C(\delta).
\end{align}
Finally, taking $\delta>0$ sufficiently small and absorbing, 
the claim follows.
\end{proof}

\begin{prop}
\label{ltestds3} 
We have
\begin{align*}
    \sup_{t\in[0,T)}\left\Vert\partial_s^3\theta\right\Vert_{L^2(0,L)}\leq C,\qquad
    \sup_{t\in[0,T)}\left\Vert\partial_s^3\rho\right\Vert_{L^2(0,L)}\leq C.
\end{align*}
\end{prop}

\begin{proof} 
Again, we use a Gronwall argument. To that end, we define
\begin{align*}
    \varphi(t):=\frac{1}{2}\int_0^L\left((\partial_s^3\theta)^2+(\partial_s^3\rho)^2\right)\intd s
\end{align*}
for $t\in[0,T)$ and obtain after integration by parts
\begin{align}
\label{dtphi3}
    \qquad\frac{\intd }{\intd t}\varphi(t)+\varphi(t)=\!-\int_0^L\!\!\partial_s^4\theta\,\partial_t\partial_s^2\theta+\partial_s^4\rho\,\partial_t\partial_s^2\rho\intd s - \frac{1}{2}\int_0^L \!\!\partial_s^4\theta\partial_s^2\theta + \partial_s^4\rho\partial_s^2\rho\intd s.
\end{align}
Differentiating twice the evolution equations \eqref{eq:flow equation} with respect to $s$ yields
\begin{align*}
    \partial_s^2\partial_t\theta
    &=3\beta^\prime(\rho)\partial_s\rho\partial_s^3\theta+\beta(\rho)\partial_s^4\theta+3\beta^{\prime\prime}(\rho)(\partial_s\rho)^2\partial_s^2\theta
    +3\beta^\prime(\rho)\partial_s^2\rho\partial_s^2\theta\\
    &\quad+\beta^{\prime\prime\prime}(\rho)(\partial_s\rho)^3(\partial_s\theta-c_0)
    +3\beta^{\prime\prime}(\rho)\partial_s^2\rho\partial_s\rho(\partial_s\theta-c_0)
    +\beta^\prime(\rho)\partial_s^3\rho(\partial_s\theta-c_0)\\
    &\quad-\lambda_{\theta1}\sin\theta(\partial_s\theta)^2+\lambda_{\theta1}\cos\theta\partial_s^2\theta
    +\lambda_{\theta2}\cos\theta(\partial_s\theta)^2+\lambda_{\theta2}\sin\theta\partial_s^2\theta,\\
    \partial_s^2\partial_t\rho
    &=\mu\partial_s^4\rho-\frac{1}{2}\beta^{\prime\prime\prime}(\rho)(\partial_s\rho)^2(\partial_s\theta-c_0)^2-\frac{1}{2}\beta^{\prime\prime}(\rho)\partial_s^2\rho(\partial_s\theta-c_0)^2\\
    &\quad-2\beta^{\prime\prime}(\rho)\partial_s\rho(\partial_s\theta-c_0)\partial_s^2\theta
    -\beta^\prime(\rho)(\partial_s^2\theta)^2-\beta^\prime(\rho)(\partial_s\theta-c_0)\partial_s^3\theta.
\end{align*}
By continuity, Proposition \ref{prop:ltest1} and \eqref{eq:Jbeta}, we find that $|\beta^{(k)}(\rho)|$ is bounded by $C$ for $k=0, \dots , 3$. Moreover, due to the embedding $W^{2,2}(0,L)\hookrightarrow C^1([0,L])$, Proposition \ref{prop:ltest1} and Proposition \ref{ltestds2} we know that
$\sup_{s\in[0,L]}\left\vert\partial_s\theta\right\vert$
is bounded by $C$ for all $t\in[0,T)$. 
The same applies to $\sup_{s\in[0,L]}\left\vert\partial_s\rho\right\vert$.
This implies that after plugging the last two equations into \eqref{dtphi3} we get
\begin{align}
\label{dtphi4}
    \frac{\intd }{\intd t}\varphi(t) + \varphi(t)
    \leq &-\inf_{J}\beta(\rho)\int_0^L(\partial_s^4\theta)^2\intd s-\mu\int_0^L(\partial_s^4\rho)^2\intd s +C R(t),
\end{align}
where $R(t)$ is a sum of terms of the form
\begin{align}\label{eq:dtphi4,5}
    \qquad
    \int_0^L\!\left\vert\partial_s^4f \right\vert\!\left\vert\partial_s^3g\right\vert\!\intd s,\;
    \int_0^L\!\left\vert\partial_s^4f \right\vert\!\left\vert\partial_s^2g\right\vert\!\intd s,\;
    \int_0^L\!\left\vert\partial_s^4f \right\vert\!\intd s,\; 
    \int_0^L \!\left\vert\partial_s^4 f\right\vert\! \left\vert\partial_s^2f\right\vert\! \left\vert\partial_s^2g\right\vert\!\intd s
\end{align}
for $f,g\in \{\theta,\rho\}$. For the first term, we use Young's inequality, integration by parts and Proposition \ref{ltestds2} and estimate
\begin{align*}
    \int_0^L\left\vert\partial_s^4f \right\vert\left\vert\partial_s^3g\right\vert\intd s
    &\leq\delta\int_0^L\left(\partial_s^4f \right)^2\intd s
    +C(\delta)\int_0^L\left(\partial_s^3g\right)^2\intd s\\
    &\leq \delta\int_0^L\left(\partial_s^4f \right)^2\intd s
    +\delta\int_0^L\left(\partial_s^4g \right)^2\intd s
    +C(\delta).
\end{align*}
Similarly, for terms of the second and third type we have
\begin{align*}
    \int_0^L\left(\left\vert\partial_s^4f \right\vert\left\vert\partial_s^2g\right\vert + \left\vert \partial_s^4f\right\vert\right)\intd s
    &\leq\delta\int_0^L\left(\partial_s^4f \right)^2\intd s
    +C(\delta).
\end{align*}
The terms of the fourth type in \eqref{eq:dtphi4,5} can be controlled by interpolation, cf.\ \cite[p.~233]{Brezis} (using $\int_0^L\partial_s^k f \intd s=0$ for $k\geq 2$, $f\in\{\theta,\rho\}$) to estimate
\begin{align*}
    &\int_0^L\left\vert\partial_s^4f \right\vert\left\vert\partial_s^2f\right\vert \left\vert\partial_s^2g\right\vert\intd s
    \leq\left\Vert\partial_s^4f\right\Vert_{L^2(0,L)}\left\Vert\partial_s^2f\right\Vert_{L^4(0,L)}\left\Vert\partial_s^2g\right\Vert_{L^4(0,L)}\\
    &\qquad\leq C\left\Vert\partial_s^4f\right\Vert_{L^2(0,L)}
    \left\Vert\partial_s^4f\right\Vert_{L^2(0,L)}^\frac{1}{8}\left\Vert\partial_s^2f\right\Vert_{L^2(0,L)}^\frac{7}{8}\left\Vert\partial_s^4g\right\Vert_{L^2(0,L)}^\frac{1}{8}\left\Vert\partial_s^2g\right\Vert_{L^2(0,L)}^\frac{7}{8}\\
    &\qquad\leq \delta\left\Vert\partial_s^4f\right\Vert_{L^2(0,L)}^2+\delta\left\Vert\partial_s^4g\right\Vert_{L^2(0,L)}^2
    +C(\delta)
\end{align*}
as above by Proposition \ref{ltestds2} and Young's inequality for $\delta>0$ to be chosen.
This way, we have estimated the remainder term $R(t)$ in \eqref{dtphi4}. 
Consequently, if we now use $\inf_J\beta\geq C$ by \eqref{eq:beta>0} and choose $\delta>0$ sufficiently small, after absorbing we get
\begin{align*}
    \frac{\intd }{\intd t}\varphi(t)+\varphi(t)\leq C.
\end{align*}
With Gronwall's lemma, we conclude that $\varphi(t)\leq C$ for all $t\in[0,T)$.
\end{proof}

\begin{rem}\label{rem:global_Wm2_bound}
It is also possible to bound $\sup_{t\in [0,T)}||(\theta,\rho)||_{W^{m,2}(0,L)}\leq C(m)$, $m\in \NN,$ by essentially the same arguments as in Proposition \ref{ltestds3}. 
\end{rem}

\subsection{Proof of Theorem \ref{longtimeex}}\label{sec:thm1.3}

Using the previous propositions, we are now able to prove long-time existence.
\begin{proof}[Proof of Theorem \ref{longtimeex}]
Let $0<T_{\textup{max}}\leq\infty$ be the maximal existence time of the smooth solution of \eqref{eq:flow equation} with initial datum $(\theta,\rho)(0,\cdot)=(\theta_0, \rho_0)$ in $C^{1+\alpha}([0,L])$. 
We want to show that the assumption $T_{\textup{max}}<\infty$ yields a contradiction as we can extend the solution past $T_{\textup{max}}$. 

Assume $T_{\textup{max}}<\infty$ and let $\delta \in (0, T_{\textup{max}})$. We have $(\theta, \rho)\in C^\infty([\delta,T_{\textup{max}})\times [0,L])$. 
By Propositions \ref{prop:ltest1} and \ref{ltestds3}, there is a constant $C>0$ depending only on $(\theta,\rho)(\delta,\cdot)$ and the model parameters such that
\begin{align}
\label{eq:boundW32}
    \sup_{t\in[\delta,T_{\textup{max}})}\left\Vert\theta\right\Vert_{W^{3,2}(0,L)}\leq C
    \quad\text{ and }\quad
    \sup_{t\in[\delta,T_{\textup{max}})}\left\Vert\rho\right\Vert_{W^{3,2}(0,L)}\leq C.
\end{align}
We consider an arbitrary sequence $(t_n)_{n\in \NN}$ with $t_n\nearrow T_{\textup{max}}$. Let $\tilde{\alpha}\in (0,\frac{1}{2})$. The bounds in \eqref{eq:boundW32} together with the compact embedding $W^{3,2}(0,L)\hookrightarrow C^{2+\tilde\alpha}([0,L])$ and the closedness of $h^{2+\tilde\alpha}([0,L])$ yield the existence of $\tilde\theta_0\in h^{2+\tilde\alpha}([0,L])$ and $\tilde{\rho}_0\in h^{2+\tilde\alpha}([0,L])$ such that after passing to a subsequence
\begin{align}\label{eq:LTE conv below1}
    \theta(t_n,\cdot)\to\tilde\theta_0 
    \quad\text{ and }\quad
    \rho(t_n,\cdot)\to\tilde\rho_0 \quad \text{in } C^{2+\tilde\alpha}([0,L]) \quad \text{as }n\to \infty.
\end{align}
Suppose $t_n'\nearrow T_{\textup{max}}$ denotes another sequence such that \eqref{eq:LTE conv below1} holds with $t_n'$ instead of $t_n$ and $(\tilde\theta_0', \tilde\rho_0')$ instead of $(\tilde\theta_0, \tilde\rho_0)$. Using \eqref{eq:flow equation}, Propositions \ref{prop:ltest1}, \ref{lmultbound} and \eqref{eq:boundW32}, we find $\left\Vert \partial_t \theta\right\Vert_{W^{1,2}(0,L)}\leq C$ for $t\in [\delta, T_{\textup{max}})$, where $C>0$ again only depends on the model parameters and $(\theta, \rho)(\delta,\cdot)$. Using that $W^{1,2}(0,L)\hookrightarrow C^{\tilde\alpha}([0,L])$, we find
\begin{align}\label{eq:LTE conv limitsequal}
    \left\Vert \theta(t_n,\cdot) - \theta(t_n',\cdot)\right\Vert_{C^{\tilde\alpha}([0,L])}\leq \Big\Vert\int_{t_n}^{t_n'} \partial_t \theta(\tau, \cdot)\intd \tau \Big\Vert_{C^{\tilde\alpha}([0,L])} \leq C|t_n-t_n'|,
\end{align}
after suitably modifying $C>0$.
Taking $n\to\infty$ yields $\tilde\theta_0'=\tilde\theta_0$. With exactly the same argument we find $\tilde\rho_0'=\tilde\rho_0$. Consequently,
\begin{align}\label{eq:LTE conv below2}
    \theta(t,\cdot)\to\tilde\theta_0 \quad\text{ and }\quad
    \rho(t,\cdot)\to\tilde\rho_0 \quad\text{ in } C^{2+\tilde\alpha}([0,L])\quad \text{as }t\nearrow T_{\textup{max}}.
\end{align}
It is clear that $(\tilde\theta_0,\tilde\rho_0)$ satisfies \eqref{eq:bcid}. 
Thus, Theorem \ref{thm:Main STE} and Remark \ref{rem:ste+} yield the existence of $\tilde{T}>0$ and a unique solution $(\tilde\theta,\tilde\rho)\in C^\infty\big((0,\tilde{T})\times[0,L]\big)$ of \eqref{eq:flow equation} with
\begin{align}
\label{eq:initialtilde}
 \tilde\theta(t,\cdot)\to\tilde\theta_0 
    \quad\text{ and }\quad
    \tilde\rho(t,\cdot)\to\tilde\rho_0 \quad\text{ in } C^{2+\tilde\alpha}([0,L])\quad \text{as }t\searrow 0.\end{align}
We take this solution to extend the solution $(\theta,\rho)$ past $T_{\textup{max}}$ by defining 
\begin{align*}
    (\bar{\theta},\bar{\rho})(t,s):=
    \begin{cases}
    (\theta,\rho)(t,s) &\text{ for } t\in(0,T_{\textup{max}}),\\
    (\tilde\theta_0,\tilde\rho_0)(s) &\text{ for } t=T_{\textup{max}},\\
    (\tilde\theta,\tilde\rho)(t-T_{\textup{max}},s) &\text{ for } t\in(T_{\textup{max}},T_{\textup{max}}+\tilde{T}).
    \end{cases}
\end{align*}
We now claim that 
\begin{align}\label{eq:LTE bar smooth}
(\bar{\theta},\bar{\rho})\in C^\infty((0,T_{\textup{max}}+\tilde{T})\times[0,L]).
\end{align}
Fix $T\in (T_{\textup{max}},T_{\textup{max}}+\tilde{T})$.
Examining $t\nearrow T_{\textup{max}}$ and using \eqref{eq:LTE conv below2}, respectively $t\searrow T_{\textup{max}}$ and using \eqref{eq:initialtilde}, we conclude 
\begin{align*} 
(\bar{\theta},\bar{\rho})\in\textup{BUC}\left([\delta,T];C^{2+\tilde\alpha}([0,L])\right).
\end{align*}
By \eqref{eq:flow equation}, we may write $\partial_t(\theta,\rho)$ in terms of $\partial_s^2(\theta,\rho),\partial_s(\theta,\rho)$ and $(\theta,\rho)$ on $(0, T_\textup{max})$ and similarly for $\partial_t(\tilde\theta, \tilde\rho)$ on $(0, \tilde{T})$. Hence $\partial_t(\bar{\theta}, \bar{\rho})$ exists on $[\delta, T_\textup{max})\cup (T_\textup{max}, T]$ and possesses a continuous extension at $t=T_\textup{max}$, so that we find
\begin{align*} 
\partial_t(\bar{\theta},\bar{\rho})\in\textup{BUC}\left([\delta,T];C^{\tilde\alpha}([0,L])\right).
\end{align*}

Consequently, we obtain
\begin{align}\label{eq:LTE BUC reg delta}
(\bar{\theta},\bar{\rho})\in\textup{BUC}^1\left([\delta,T];C^{\tilde\alpha}([0,L])\right)\cap\textup{BUC}\left([\delta,T];C^{2+\tilde\alpha}([0,L])\right).
\end{align}

By \eqref{eq:STE BUC regularity}, for every $T'<T_{\textup{max}}$ the solution $(\theta, \rho)$ satisfies
\begin{align}\label{eq:LTE BUC Tmax}
    \quad(\theta,\rho)\in\textup{BUC}_{1-\bar{\eta}}^1\left([0,T']; h^{\bar{\alpha}}\left([0,L]\right)\right)\cap \textup{BUC}_{1-\bar{\eta}}\left([0,T'];h^{2+\bar{\alpha}}\left([0,L]\right)\right),
\end{align}
where $\bar{\alpha}\in (0,1)$ and $\bar{\eta}\in (\frac12,1)$ are chosen such that $2\bar{\eta}+\bar{\alpha}=1+\alpha$. Since the time weight $t^{1-\bar{\eta}}$ only plays a role near $t=0$, from \eqref{eq:LTE BUC reg delta} and \eqref{eq:LTE BUC Tmax} we conclude that
\begin{align*}
(\bar{\theta},\bar{\rho})\in\textup{BUC}^1_{1-\bar{\eta}}\left([0,T];h^{\hat\alpha}([0,L])\right)\cap\textup{BUC}_{1-\bar{\eta}}\left([0,T];h^{2+\hat\alpha}([0,L])\right),
\end{align*}
for $0<\hat\alpha<\min\{\bar{\alpha}, \tilde{\alpha}\}$, using Lemma \ref{lem:Höldemb}. If we transfer this back to the periodic setting as in Section \ref{subsec:trafoper}, we may apply Lemma \ref{smoothing} to deduce that $(\bar{\theta}, \bar{\rho})\in C^\infty((0,T)\times [0,L])$. Since $T\in (T_{\textup{max}}, T_{\textup{max}}+\tilde{T})$ was arbitrary, we conclude that \eqref{eq:LTE bar smooth} is satisfied. 

Consequently, we have extended the solution $(\theta,\rho)$ smoothly past $T_{\textup{max}}$, a contradiction, so $T_{\textup{max}}=\infty$ has to hold.
\end{proof}

\section{Convergence result}\label{sec:longtimebehavior}
In this section, we prove the convergence result, Theorem \ref{thm:convergence}.
Our main ingredient is a constrained version of the {\L}ojasiewicz--Simon inequality \cite{ConstrLoja}. The constraint is given by the zero set of the functional
\begin{align}
\label{eq:defG}
    \mathcal{G}(\theta, \rho) = \Big( \int_0^L \cos\theta\intd s,\; \int_0^L \sin \theta\intd s,\; \int_0^L \rho \intd s - m\Big),
\end{align}
cf.\ \eqref{eq:constraint conservation}, where $m\in \R$ corresponds to the fixed total mass, determined by the initial density, i.e.\ $m=\int_0^L \rho_0\intd s$. We consider the Banach space of periodic Sobolev functions
\begin{align}
W^{k,2}_{\mathrm{per}}(0,L) := \{ u\in W^{k,2}(0,L) : \partial_s^\ell u(L)=\partial_s^\ell u(0) \text{ for }\ell = 0,\dots, k-1\},\; k\in \NN.
\end{align}
By the choice of the Lagrange multipliers, cf.\ \eqref{eq:dt lambda theta} and \eqref{eq:dt lambda rho}, the gradient flow remains in the closed set
\begin{align}\label{eq:defB1}
\sB := \{ (\theta,\rho)\in W^{2,2}_{\mathrm{per}}(0,L;\R^2)+(\phi,0) : \sG(\theta,\rho)=0\}
\end{align}
for all $t\geq 0$ with $\phi$ as in \eqref{eq:defphi}.
To apply the results in \cite{ConstrLoja}, we need to work in Banach spaces which is why we consider the shifted energies
\begin{align}
    &E_\mu(u,\rho) = \sE_\mu(\phi+u, \rho),\quad
    G(u,\rho) = \sG(\phi+u, \rho) \quad \text{for }(u,\rho)\in W^{2,2}_{\mathrm{per}}(0,L;\R^2).
\end{align}
 Here we work only in the domain of the $L^2$-gradient of the functionals, and not in the energy space $W^{1,2}(0,L)$. 
 By a direct computation, the $L^2$-gradients of $E_\mu$ (with $\nabla E_{\mu}(u, \rho) = (\nabla_u E_\mu(u,\rho), \nabla_\rho E_\mu(u,\rho)))$ and of the components of $G$ are given by
\begin{align}
    &\nabla_u E_{\mu}(u, \rho) = -\partial_s\big(\beta(\rho)(\partial_su+\partial_s\phi-c_0)\big),\\
    & \nabla_\rho E_\mu(u,\rho) = -\mu\partial_s^2\rho +\frac{1}{2}\beta'(\rho)(\partial_su+\partial_s\phi-c_0)^2,\\
    &\nabla G ^1(u, \rho) = \left(-\sin(u+\phi),0\right),\;
    \nabla G ^2(u, \rho) = \left(\cos(u+\phi),0\right), \;
    \nabla G ^3(u, \rho) = \left(0, 1\right).
\end{align}

\subsection{Analyticity}

We first discuss analyticity properties of the energy and the constraint. 
A concise overview of the relevant properties of analytic functions on Banach spaces can be found in \cite[Section 2.1]{ConstrLoja}. In the following,
$C_{\mathrm{per}}([0,L])$ denotes the Banach space of $L$-periodic continuous functions on $[0,L]$, equipped with the supremum norm.

\begin{lem}\label{lem:analyticity}
    Let $\beta\colon \R\to\R$ be analytic. The maps
    \begin{align}
        &W^{1,2}_{\mathrm{per}}(0,L)\to C_{\mathrm{per}}([0,L]),\; \rho\mapsto \beta(\rho),\\
    &W^{k,2}_{\mathrm{per}}(0,L)\to L^2(0,L),\; \theta\mapsto \partial_s^{k}\theta \text{ for }k=1,2
    \end{align}
   are analytic. Moreover, $E_\mu, G$ and their gradients $\nabla E_\mu, \nabla G^j\colon W^{2,2}_{\mathrm{per}}(0,L;\R^2)\to L^2(0,L;\R^2)$ are analytic for $j=1, 2, 3$.
\end{lem}

\begin{proof}
The first part follows from the Sobolev embedding $W^{1,2}_{\mathrm{per}}(0,L)\to C_{\mathrm{per}}([0,L])$ and since the Nemytskii operator $C_{\mathrm{per}}([0,L])\ni u\mapsto \beta(u)\in C_{\mathrm{per}}([0,L])$ is analytic by  \cite[Theorem 6.8]{AppellZabrejko}. 

The rest follows using that the composition and sum of analytic maps is again analytic and that any bounded multilinear map is analytic, see \cite[Section 2.1]{ConstrLoja}.
\end{proof}

\subsection{The constrained \texorpdfstring{{\L}ojasiewicz--Simon}{Lojasiewicz-Simon} gradient inequality}

Besides analyticity, we will also need to verify certain Fredholm and compactness properties of the derivatives of $\nabla{E}_\mu$ and $\nabla{G}$, cf.\ \cite[Corollary 5.2]{ConstrLoja}. 
First, we compute these derivatives.
Let $(u, \rho), (v, \sigma)\in W^{2,2}_{\mathrm{per}}(0,L;\R^2)$. 
Then we have
\begin{align}
    (\nabla_u {E}_\mu)'(u,\rho)(v,\sigma) &= -\beta''(\rho)\partial_s \rho(\partial_su+\partial_s \phi-c_0)\sigma - \beta'(\rho)\partial_s^2 u~\sigma\\
    &\qquad~- \beta'(\rho)(\partial_su+\partial_s\phi-c_0)\partial_s \sigma - \partial_s\big(\beta(\rho)\partial_s v\big),\\
    (\nabla_\rho {E}_\mu)'(u,\rho)(v,\sigma) &= \frac{1}{2} \beta''(\rho)(\partial_su+\partial_s\phi-c_0)^2\sigma \\
    &\qquad~+ \beta'(\rho)(\partial_su+\partial_s\phi-c_0)\partial_sv - \mu\partial_s^2\sigma. \label{eq:nabla E'}
\end{align}
Moreover, we have
\begin{align}
    (\nabla G^1)'(u, \rho)(v,\sigma) &= (-\cos(u+\phi) v, 0), \\
    (\nabla G^2)'(u, \rho)(v,\sigma) &= (-\sin(u+\phi) v, 0), \\
    (\nabla G^3)'(u, \rho)(v,\sigma) &= (0,0). 
\end{align}
\begin{lem}\label{lem:nablaE'Fredholm}
The operator 
$
(\nabla {E}_\mu)'(u, \rho)\colon W^{2,2}_{\mathrm{per}}(0,L;\R^2)\to L^2(0,L;\R^2)    
$
is Fredholm of index zero for all $(u, \rho)\in W^{2,2}_{\mathrm{per}}(0,L;\R^2)$. 
\end{lem}
\begin{proof}
We observe that by \eqref{eq:nabla E'}, we may write
\begin{align}
    (\nabla {E}_\mu)'(u, \rho)(v, \sigma) = A(v, \sigma) + K(v,\sigma),
\end{align}
where $A(v,\sigma) = \left(-\partial_s\left(\beta(\rho)\partial_s v\right), -\mu \partial_s^2\sigma\right)$ and $K\colon W^{2,2}_{\mathrm{per}}(0,L;\R^2)\to L^2(0,L;\R^2)$ only involves first and zeroth order expressions in $(v, \sigma)$. Hence $K$ is compact by the Rellich--Kondrachov theorem. Moreover, $\beta(\rho)\in C^1([0,L])$ and $\min_{[0,L]}\beta(\rho)>0$ by Sobolev embedding. Thus, for every $f\in L^2(0,L)$ there exists a unique solution $v\in W^{2,2}_{\mathrm{per}}(0,L)$ to
$\partial_s\left(\beta(\rho) \partial_s v\right) = f$
as a consequence of the Riesz representation theorem. 
Consequently, we find that $A\colon W^{2,2}_{\mathrm{per}}(0,L;\R^2)\to L^2(0,L;\R^2)$ is an isomorphism. The claim then follows using \cite[XVII, Corollaries 2.6 and 2.7]{Lang}.
\end{proof}

For simplicity, we abuse the notation and write
\begin{align}
    \nabla \sE_\mu(\theta, \rho) := \nabla {E}_\mu (\theta-\phi, \rho), \quad
    \nabla \mathcal{G}(\theta, \rho) &:= \nabla {{G}} (\theta-\phi, \rho)\label{eq:defB}
\end{align}
for $(\theta, \rho) \in W^{2,2}_{\mathrm{per}}(0,L;\R^2)+(\phi,0)$  and $\phi$ as in \eqref{eq:defphi}. 
Moreover, setting $\Lambda(\theta, \rho):= (\lambda_{\theta1}(\theta, \rho), \lambda_{\theta2}(\theta, \rho), \lambda_\rho(\theta, \rho))$, 
we define 
\begin{align}
    \Lambda(\theta, \rho)\cdot \nabla \sG(\theta, \rho) := \lambda_{\theta1}(\theta, \rho)\nabla \sG^1(\theta,\rho)+\lambda_{\theta2}(\theta, \rho)\nabla\sG^2(\theta,\rho)+\lambda_\rho(\theta, \rho)\nabla\sG^3(\theta,\rho).
\end{align}
In this notation, the evolution in \eqref{eq:flow equation} may abstractly be written as
\begin{align}\label{eq:abstract CGF}
    \partial_t(\theta,\rho)=-\nabla \sE_\mu(\theta,\rho)- \Lambda(\theta, \rho)\cdot \nabla \sG(\theta, \rho).
\end{align}
where $(\theta,\rho)\colon[0,T)\times[0,L]\to\R^2$.
This way, stationary solutions to \eqref{eq:abstract CGF} (i.e.\ solutions to \eqref{eq:stationary}) are precisely the constrained critical points of $\mathcal{E}$ subject to the constraint $\sG=0$ by the method of Lagrange multipliers.
\begin{thm}[{Constrained {\L}ojasiewicz--Simon gradient inequality}]\label{thm:loja} Let $\beta$ be real analytic and suppose that $(\bar{\theta}, \bar{\rho})\in \mathcal{B}$ (cf.\ \eqref{eq:defB1}) is a  critical point of $\sE_\mu$ subject to the constraint $\sG=0$. 
Then there exist $C, r>0$ and $\vartheta \in (0,\frac{1}{2}]$ such that for all $(\theta, \rho) \in \mathcal{B}$ with $\Vert (\theta, \rho)-(\bar{\theta}, \bar{\rho})\Vert_{W^{2,2}(0,L)} \leq r$ we have
\begin{align}
    |\sE_\mu(\theta, \rho)-\sE_\mu(\bar{\theta}, \bar{\rho})|^{1-\vartheta}
    &\leq C\Vert \nabla \sE_\mu(\theta,\rho)+\Lambda(\theta, \rho)\cdot \nabla \sG(\theta, \rho)\Vert_{L^2(0,L)}.
\end{align}
\end{thm}
\begin{proof}
We verify that assumptions (i)--(vi) in \cite[Corollary 5.2]{ConstrLoja} are satisfied for the energy functional ${E}_\mu \colon W^{2,2}_{\mathrm{per}}(0,L;\R^2)\to\R$ subject to the constraint $G=0$ at the constrained critical point $(\bar{u}, \bar{\rho}):= (\bar{\theta}-\phi, \bar\rho)$. 
Assumption (i) is clearly satisfied. Assumptions (ii) and (iv), and also analyticity of ${E}_\mu$ and $G$ follow from Lemma \ref{lem:analyticity}, whereas (iii) is precisely the statement of Lemma \ref{lem:nablaE'Fredholm}. Assumption (v) is satisfied by \eqref{eq:nabla E'} and using the Rellich--Kondrachov theorem.  
For (vi), it only remains to show that $\nabla G^1(\bar{u}, \bar{\rho}), \nabla G^2(\bar{u}, \bar{\rho}),\nabla G^3(\bar{u}, \bar{\rho})$ are linearly independent, which can be verified using $\int_0^L\cos\bar\theta\intd s = \int_0^L\sin\bar\theta\intd s=0$.

By \cite[Corollary 5.2]{ConstrLoja}, ${E}_\mu\vert_{\{G=0\}}$ satisfies the constrained {\L}ojasiewicz--Simon gradient inequality near $(\bar{u},\bar\rho)$, which immediately transfers to $\sE_\mu\vert_{\{\sG=0\}}$ near $(\bar\theta, \bar\rho)$ as in the statement. The explicit form of the projected gradient in terms of the Lagrange multipliers follows from \cite[Proposition~3.3]{ConstrLoja}.
\end{proof}

\subsection{Convergence}
\label{sec:convergence}

Equipped with the {\L}ojasiewicz--Simon gradient inequality as a powerful functional analytic tool, we can now prove Theorem \ref{thm:convergence}.

\begin{proof}[{Proof of Theorem \ref{thm:convergence}}]
Since we are only interested in the long-time behavior, we may without loss of generality assume $\theta, \rho\in C^\infty([0,\infty)\times [0,L])$.

\textit{Step 1: Subconvergence.} We first prove that for any sequence $t_n\to\infty$, there exist a subsequence $(t'_{n})_{n\in \NN}$ and a stationary solution $(\theta_\infty, \rho_\infty)\in C^{2+\tilde\alpha}([0,L])$ such that $\lim_{n\to\infty}(\theta(t'_n), \rho(t'_n))= (\theta_\infty, \rho_\infty)$ in $C^{2+\tilde\alpha}([0,L])$ for all $\tilde\alpha\in (0,\frac12)$.

Let $(t_n)_{n\in\NN}$ be any sequence with $t_n\to\infty$. First, note that the constant $C$ in \eqref{eq:boundW32} does not depend on the existence time, so \eqref{eq:boundW32} remains valid for $T_{\textup{max}}=\infty$. The compact embedding $W^{3,2}(0,L)\hookrightarrow C^{2+\tilde{\alpha}}([0,L])$, where $\tilde{\alpha}\in (0, \frac12)$, and a subsequence argument yields that there exists $(t_n')_{n\in \NN}$, a subsequence of $(t_n)_{n\in\NN}$, and $(\theta_\infty, \rho_\infty)$ such that for all $\tilde{\alpha}\in (0,\frac{1}{2})$ we have
\begin{align}\label{eq:subconvergence C2alpha}
    \lim_{n\to\infty}(\theta(t_n'), \rho(t_n')) = (\theta_\infty,\rho_\infty) \quad \text{in } C^{2+\tilde\alpha}([0,L]).
\end{align}
It remains to prove that $(\theta_\infty, \rho_\infty)$ is stationary, i.e.\ it satisfies \eqref{eq:stationary}. We consider 
\begin{align}
    \varphi(t) = \frac{1}{2}\int_0^L \left((\partial_t \theta)^2+(\partial_t\rho)^2\right)\intd s, \quad t\geq 0,
\end{align}
which is bounded by Proposition \ref{dtLinfL2}, and thus by \eqref{eq:gronwall linear} its derivative $\frac{\intd}{\intd t}\varphi$ is bounded from above. Since $\varphi$ is integrable on $[0, \infty)$ by Proposition \ref{prop:ltest1}, this is sufficient to guarantee that $\lim_{t\to\infty}\varphi(t)=0$. In particular by \eqref{eq:subconvergence C2alpha}, we have $\partial_t \theta(t_n',\cdot)\to 0$, $ \partial_t \rho(t_n',\cdot)\to 0$ in $C^{\tilde{\alpha}}([0,L])$ for all $\tilde{\alpha}\in (0,\frac12)$, so $(\theta_\infty, \rho_\infty)$ is stationary, thus a solution of \eqref{eq:stationary}.

\textit{Step 2: Full convergence.} 
By Step 1, there exists a sequence $t_n\to\infty$ and a stationary solution $(\theta_\infty,\rho_\infty)$ of \eqref{eq:flow equation} such that $(\theta_\infty,\rho_\infty) = \lim_{n\to\infty} (\theta(t_n),\rho(t_n))$ in $C^{2+\tilde{\alpha}}([0,L])$ for all $\tilde{\alpha}\in (0,\frac{1}{2})$.  

By \eqref{eq:energy decay}, we have $\sE_\mu(\theta(t),\rho(t))\searrow \sE_\mu(\theta_\infty, \rho_\infty)$ as $t\to\infty$. In fact, we may assume $\sE_\mu(\theta(t),\rho(t))> \sE_\mu(\theta_\infty, \rho_\infty)$ for all $t\in [0,\infty)$, since otherwise the solution must be eventually constant, hence converges trivially.

The stationary solution $(\theta_\infty, \rho_\infty)$ is a constrained critical point as observed directly after \eqref{eq:abstract CGF}. Thus, we may apply Theorem \ref{thm:loja} to $(\theta_\infty, \rho_\infty)$. Let $C, r>0$ and $\vartheta\in (0,\frac{1}{2}]$ be as in Theorem \ref{thm:loja}. After passing to a subsequence, we may assume $\Vert (\theta(t_n), \rho(t_n))-(\theta_\infty, \rho_\infty)\Vert_{W^{2,2}(0,L)}< r$ for all $n\in \NN$. We define
\begin{align}\label{eq:conv proof 1}
    \qquad\tau_n := \sup\{\tau\geq t_n\mid \Vert (\theta(t), \rho(t))-(\theta_\infty, \rho_\infty)\Vert_{W^{2,2}(0,L)}< r \text{ for all }t\in [t_n,\tau]\}.
\end{align}
Since the solution is smooth, we find that $\tau_n>t_n$ for all $n\in \NN$. Define $\mathcal{H}(t):= \left(\sE_\mu(\theta(t), \rho(t))-\sE_\mu(\theta_\infty, \rho_\infty)\right)^\vartheta$. By a direct computation and using Theorem \ref{thm:loja} we have
\begin{align}
    -\frac{\intd}{\intd t}\mathcal{H} &= \vartheta\Big(\sE_\mu(\theta, \rho)-\sE_\mu(\theta_\infty, \rho_\infty)\Big)^{\vartheta-1}\Big(-\frac{\intd}{\intd t}\sE_\mu(\theta, \rho)\Big)\\
    &= \vartheta\Big(\sE_\mu(\theta, \rho)-\sE_\mu(\theta_\infty, \rho_\infty)\Big)^{\vartheta-1}\Vert\nabla \sE_\mu(\theta,\rho)\Vert_{L^2(0,L)}\Vert\partial_t(\theta,\rho)\Vert_{L^2(0,L)}\\
    &\geq \frac{\vartheta}{C} \Vert\partial_t(\theta,\rho)\Vert_{L^2(0,L)}\label{eq:conv proof 3}
\end{align}
on $[t_n, \tau_n)$.
After integration, for $t'\in [t_n, \tau_n)$ we obtain
\begin{align}\label{eq:conv proof 4}
    \qquad\Vert (\theta(t'), \rho(t'))-(\theta(t_n),\rho(t_n))\Vert_{L^2(0,L)} \leq \int_{t_n}^{t'}\Vert\partial_t(\theta,\rho)\Vert_{L^2}\intd t \leq \frac{C}{\vartheta} \mathcal{H}(t_n) \to 0,
\end{align}
as $n\to\infty$.
We now assume that all of the $\tau_n$ are finite. By continuity, \eqref{eq:conv proof 4} also remains valid for $t'=\tau_n$. By Step 1, after passing to a subsequence, we may assume that $(\theta(\tau_n), \rho(\tau_n))\to (\hat{\theta}, \hat{\rho})$ in $C^{2+\tilde{\alpha}}([0,L])$ as $n \to\infty$ for all $\tilde\alpha\in (0,\frac{1}{2})$. By continuity and \eqref{eq:conv proof 1}, we find that 
\begin{align}\label{eq:W22r}
    \Vert (\hat{\theta}, \hat{\rho})-(\theta_\infty, \rho_\infty)\Vert_{W^{2,2}(0,L)} = r>0,
\end{align}
whereas on the other hand by \eqref{eq:conv proof 4} with $t'=\tau_n$, we have $\Vert (\hat{\theta}, \hat{\rho})-(\theta_\infty, \rho_\infty)\Vert_{L^2(0,L)} = \lim_{n\to\infty} \Vert (\theta(\tau_n), \rho(\tau_n))-(\theta(t_n), \rho(t_n))\Vert_{L^2(0,L)} =0$, a contradiction to \eqref{eq:W22r}.

Consequently, there exist some $n_0$ with $\tau_{n_0}=\infty$, and thus \eqref{eq:conv proof 3} is satisfied on $[t_{n_0}, \infty)$, which implies that $t\mapsto \Vert \partial_t(\theta(t), \rho(t))\Vert_{L^2(0,L)} \in L^1(0,\infty)$. Therefore, $(\theta(t), \rho(t))$ is Cauchy in $L^2(0,L)$ as $t\to\infty$. Hence the limit $\lim_{t\to\infty}(\theta(t), \rho(t))$ exists in $L^2(0,L)$ and thus necessarily equals $(\theta_\infty, \rho_\infty)$. From Step 1 and a subsequence argument, it follows that we have $\lim_{t\to\infty}(\theta(t), \rho(t))=(\theta_\infty, \rho_\infty)$ in $C^{2+\tilde{\alpha}}([0,L])$ for all $\tilde{\alpha}\in (0,\frac{1}{2})$.
\end{proof}

\begin{rem}
\begin{enumerate}[(i)]
    \item In Step 1 above, we have not used that $\beta$ is analytic. Thus, subconvergence also holds if merely $\beta\in C^\infty(\R)$.
    \item Together with stronger global bounds on the solution as in Remark \ref{rem:global_Wm2_bound}, it can be shown that the convergence in Theorem \ref{thm:convergence} is smooth. The same applies to the subconvergence in Step 1 of its proof.
\end{enumerate}
\end{rem}

\appendix
\section{Function Spaces}
\label{appfuncspace}

Already in the introduction in Theorem \ref{thm:Main STE} little Hölder functions appear. Later in Chapter \ref{ste} we work with time-weighted little Hölder spaces and in the proof of Lemma \ref{smoothing} additionally the parabolic Hölder spaces are used. 
Here we collect the definitions of these spaces and list some properties.
\subsection{Little Hölder functions}
For $\alpha\in\RR_+\setminus\ZZ$, we denote by $\lfloor\alpha\rfloor$ the largest integer less than $\alpha$ and $\{\alpha\}:=\alpha-\lfloor\alpha\rfloor \in (0,1)$. Let $k\in\NN$.
The space of $L$-periodic $\RR^k$-valued functions over $\RR$ which are continuous and $\lfloor\alpha\rfloor$-times continuously differentiable is denoted by $C^{\lfloor\alpha\rfloor}_L(\RR;\RR^k)$. The \textit{periodic Hölder space with exponent $\alpha$} is defined as
\begin{align*}
    C^\alpha_L(\RR;\RR^k)
    &:=\Big\lbrace f\!\in\! C^{\lfloor\alpha\rfloor}_L(\RR;\RR^k): \left[f^{\left(\lfloor\alpha\rfloor\right)}\right]_{\alpha,\RR}\hspace{-.5em}:=\sup_{\substack{x,y\in\RR, \\ x\neq y}}\frac{\left\lvert f^{\left(\lfloor\alpha\rfloor\right)}(x)-f^{\left(\lfloor\alpha\rfloor\right)}(y)\right\rvert}{|x-y|^{\{\alpha\}}}<\infty\Big\rbrace
\end{align*}
and equipped with the norm
$ \left\Vert f \right\Vert_{C^\alpha(\RR)}:=\left\Vert f \right\Vert_{C^{\lfloor \alpha \rfloor}(\RR)}+\left[f^{\left(\lfloor\alpha\rfloor\right)}\right]_{\alpha,\RR}$. Therewith, we define the \textit{periodic little Hölder space with exponent $\alpha$} over $\RR$ by
\begin{align}
\label{eq:deflitH}
    \qquad h^\alpha_L(\RR;\RR^k):=\Big\lbrace f\in C^\alpha_L(\RR;\RR^k):\,\lim_{\delta\to0}\hspace{-.5em}\sup_{\substack{x,y\in\RR,\\0<|x-y|<\delta}}\hspace{-.5em}\frac{\left\lvert f^{\left(\lfloor \alpha \rfloor\right)}(x)-f^{\left(\lfloor \alpha \rfloor\right)}(y)\right\rvert}{|x-y|^{\{\alpha\}}}=0\Big\rbrace.
\end{align}
This is a closed subspace of $C^\alpha_L(\RR;\RR^k)$ and therefore a Banach space.

In the following we collect some properties of periodic little Hölder functions. This type of functions plays a fundamental role in the classical maximal regularity theory, see for example \cite{L1995}.

\begin{lem}
\label{lem:Höldemb}\textup{\cite[Prop. 1.2]{lC2011}}
Let $\alpha_1,\alpha_2\in \RR_+\setminus\ZZ$ such that $\alpha_1<\alpha_2$. Then $h^{\alpha_1}_L(\RR)$ is the closure of $C^{\alpha_2}_L(\RR)$ in $\big(C^{\alpha_1}_L(\RR),\left\Vert\cdot\right\Vert_{C^{\alpha_1}}\big)$ and $h^{\alpha_2}_L(\RR)\hookrightarrow h^{\alpha_1}_L(\RR)$ densely.
\end{lem}

\begin{lem}[Product of periodic little Hölder functions]
\label{multlith}
Let $\alpha_1,\alpha_2\in\RR_+\setminus\ZZ$. Let $g_1\in h^{\alpha_1}_L(\RR)$ and let $g_2\in h^{\alpha_2}_L(\RR)$. Then
\begin{align*}
    g_1 g_2\in h^{\min\{\alpha_1,\alpha_2\}}_{per, L}(\RR)
\;\text{ and }\;
    \left\Vert g_1 g_2\right\Vert_{C^{\min\{\alpha_1,\alpha_2\}}}\leq C(\alpha_1,\alpha_2,L)\left\Vert g_1\right\Vert_{C^{\alpha_1}}\left\Vert g_2\right\Vert_{C^{\alpha_2}}.
\end{align*}
\end{lem}
\begin{proof}
This follows from \eqref{eq:deflitH}  and a lengthy computation.
\end{proof}

\subsection{Time-weighted BUC spaces}
\label{sec:BUC-spaces}

Let $E$ be an arbitrary Banach space. Let $\eta\in(0,1)$ and let $T>0$. 
We define 
\begin{align*}
    \textup{BUC}_{1-\eta}\left([0,T];E\right):=\Big\lbrace & \xi\in C\left((0,T];E\right):\left[t\mapsto t^{1-\eta}{\xi(t)}\right] \textup{  bounded and uniformly}\\ & \quad \qquad\textup{ continuous on }(0,T] \textup{ with }\lim_{t\to 0^+}{t^{1-\eta}}\left\Vert\xi(t)\right\Vert_{E}=0\Big\rbrace.
\end{align*}
Moreover, we define the subspace
\begin{align*}
    \textup{BUC}_{1-\eta}^1\left([0,T];E\right):=\left\lbrace \xi\in C^1\left((0,T];E\right):\,\xi,\partial_t\xi\in \textup{BUC}_{1-\eta}\left([0,T];E\right)\right\rbrace.
\end{align*}
In Section \ref{ste}, we use this space with $E$ a periodic little Hölder space. In particular, the spaces $\EE_0$ and $\EE_1$, defined in \eqref{eq:defEE0} and \eqref{eq:defEE1} are Banach spaces with the norms
\begin{align*}
    \left\Vert\xi\right\Vert_{\EE_0}&= \left\Vert\xi\right\Vert_{\EE_0([0,T])}:=\sup_{t\in(0,T]}t^{1-\eta}\left\Vert\xi(t)\right\Vert_{C^\alpha(\RR)}\quad\text{ and }\\
    \left\Vert\xi\right\Vert_{\EE_1}&=\left\Vert\xi\right\Vert_{\EE_1([0,T])}:=\sup_{t\in(0,T]}t^{1-\eta}\left(\left\Vert\partial_t\xi(t)\right\Vert_{C^\alpha(\RR)}+\left\Vert\xi(t)\right\Vert_{C^{2+\alpha}(\RR)}\right),
\end{align*}
respectively. We use that the mapping 
\begin{align*}
    \gamma\colon\EE_1([0,T])\to h^\alpha_L(\RR;\RR^k),\quad \xi(t,\cdot)\mapsto\xi(0,\cdot)
\end{align*}
is well defined, linear and continuous. This follows from the observation that for $\xi\in\EE_1([0,T])$ and $0<s<t$
\begin{equation}\label{eq:martin1}
\xi(t)-\xi(s)=\int_s^t \partial_t \xi (\tau) \intd\tau, 
\end{equation}
and hence there exists $ \lim_{t \to 0}\xi(t)$ in $h^\alpha_L(\RR)$, see \cite[Remark 2.1]{CS2001}.
We introduce the trace space $ \gamma\EE_1:=\textup{Im}(\gamma)$. In particular, $\gamma\EE_1$ does not depend on $T>0$.
In \cite[Section 5.1]{lC2011} it is shown that up to equivalent norms
\begin{align}\label{eq:tracespace}
    \gamma\EE_1=h^{2\eta+\alpha}_L(\RR;\RR^k)
\end{align}
provided $2\eta+\alpha\not\in\ZZ$.
For simplicity, we will mostly omit the index $k$.

If $\eta>\frac12$, we have continuity in space and time of the first order $s$-derivative by the following lemma. This is essential in Section \ref{ste}.

\begin{lem}
\label{appbuc1}\textup{\cite[Lemma 2.2]{CS2001}}
Let $T>0$. We have
\begin{align*}
    \EE_1([0,T])\hookrightarrow\textup{BUC}\big([0,T];h^{2\eta+\alpha}_L(\RR)\big).
\end{align*}
In particular, this implies that for $\xi\in\EE_1([0,T])$ there exists $\lim_{t \to 0}\xi(t)$ in $C^{2\eta+\alpha}_L(\RR)$ 
and $\sup_{t\in[0,T]}\left\Vert \xi\right\Vert_{C^{2 \eta+\alpha}} \leq C(T) \left\Vert\xi\right\Vert_{\EE_1}< \infty$.
\end{lem}

\begin{rem}
\label{rem:appbuc1}
Lemma \ref{appbuc1} implies that if $\eta\geq\frac{1}{2}$, for $\xi\in\EE_1([0,T])$, there exists $\lim_{t \to 0}\partial_s\xi(t)$ in $C^{\alpha}_L(\RR)$ and $\partial_s\xi\in\textup{BUC}\big([0,T];h^{\alpha}_L(\RR)\big)$.
\end{rem}

\begin{rem}\label{rem:appbuc2}
We observe that from the definition of the spaces $\EE_0([0,T])$ and $\EE_1([0,T])$ we can derive the following properties.
\begin{enumerate}[(i)]
    \item If $f \in \EE_1([0,T])$, then $\partial_s f,\partial_s^2 f \in \EE_0([0,T])$.
    \item If $f_1 \in \EE_1([0,T])$ and $f_2 \in \EE_0([0,T])$, then $f_1f_2 \in \EE_0([0,T])$.
    \item If $f\in\EE_1([0,T])$ and $g\in C^\infty(\RR)$, then $g\circ f\in\EE_1([0,T])$.
\end{enumerate}
\end{rem}

\begin{lem}\label{lem:ds2u in L2}
Let $T>0$, $\eta\in (\frac{1}{2},1)$, $\xi\in \EE_1([0,T])$. Then $\partial_s^2 \xi\in L^2([0,T]\times [0,L])$.
\end{lem}

\begin{proof}
Since the $\EE_1$-norm is finite, there exists $K>0$ such that for all $t\in [0,T]$ we have
$
\sup_{s\in \RR} |\partial_s^2\xi(t,s)| \leq K t^{\eta-1}.
$
It thus follows
\begin{align}
    \int_0^T \int_0^L (\partial_s^2\xi)^2\intd s \intd t \leq K^2 L \int_0^T t^{2\eta - 2}\intd t =  \frac{K^2L}{2\eta-1}T^{2\eta-1}<\infty. &\qedhere
\end{align}
\end{proof}

\subsection{Parabolic Hölder spaces}
\label{app:parhoelder}

At the end of Section \ref{ste} we show that the solution of \eqref{eq:flow eq u} from Proposition  \ref{prop:existence} is smooth. For this purpose we consider two additional function spaces.
First, we set for $\alpha\in\RR_+\setminus\ZZ$,
$\zeta\in[0,1)$ and $T>0$
\begin{align*}
    \textup{BUC}^{\zeta}\left([0,T];h^{\alpha}_L(\RR)\right)
    :=\Big\lbrace&\xi:[0,T]\to h^{\alpha}_L(\RR)
    \textup{ bounded and uniformly continuous}\\ &\textup{ on } [0,T],\;
    \sup_{{t_1\neq t_2\in[0,T]}}\frac{\left\Vert \xi(t_1,s)-\xi(t_2,s)\right\Vert_{C^\alpha}}{\left\vert t_1-t_2\right\vert^{\zeta}}<\infty\Big\rbrace.
\end{align*}
Moreover, we work with parabolic Hölder spaces, cf.\ \cite{LSU1988}. 
Let $l\in\RR_+\setminus\ZZ$, $a,b\in\RR$ with $a<b$, $T>0$ and let $Q_T:=(0,T)\times(a,b)$. The parabolic Hölder space $H^{\frac{l}{2},l}\left(\overline{Q_T}\right)$ is the space of all functions $f:\overline{Q_T}\to\RR$ with continuous derivatives $\partial_t^m\partial_s^nf$ for $2m+n<l$ and finite norm
\begin{align*}
    \left\Vert f\right\Vert_{H^{\frac{l}{2},l}\left(\overline{Q_T}\right)}
    :=\,&\sum_{j=0}^{\lfloor l \rfloor}\sum_{2m+n=j}\sup_{(t,s)\in{Q_T}}\left\vert\partial_t^m\partial_s^n f\right\vert\\
    &+\sum_{2m+n=\lfloor l \rfloor}\sup_{\substack{(t,s_1),(t,s_2)\in Q_T\\ s_1\neq s_2}}\frac{\left\vert \partial_t^m\partial_s^nf(t,s_1)-\partial_t^m\partial_s^nf(t,s_2)\right\vert}{\left\vert s_1-s_2\right\vert^{l-\lfloor l \rfloor}}\\
    &+\sum_{0<l-2m-n<2}\sup_{\substack{(t_1,s),(t_2,s)\in Q_T \\ t_1\neq t_2}}\frac{\left\vert \partial_t^m\partial_s^nf(t_1,s)-\partial_t^m\partial_s^nf(t_2,s)\right\vert}{\left\vert t_1-t_2\right\vert^{\frac{l-2m-n}{2}}}.
\end{align*}
The following lemma shows how $\EE_1$ relates to these
parabolic Hölder spaces.

\begin{lem}
\label{E1BUC}
Let $\alpha\in(0,1)$ and $\eta\in\left[\frac{1}{2},1\right)$ such that $2\eta+\alpha\not\in\ZZ$. Let $a,b\in\RR$ with $a<b$ and let $T>0$. Let $f\in\EE_1([0,T])$. Then 
$
    \partial_sf\in H^{\frac{\zeta}{2},\zeta}([0,T]\times[a,b])
$
for $\zeta:=\eta+\frac{\alpha}{2}-\frac{1}{2}\in(0,1)$.
\end{lem}

\begin{proof}
Since $0<\frac{\eta}{2}-\frac{\alpha}{4}+\frac{1}{4}<\eta$, \cite[Lemma 2.2 d)]{CS2001} together with \cite[Prop. 1.2 b)]{lC2011} give
\begin{align*}
    \EE_1([0,T])\hookrightarrow
   &\,\textup{BUC}^{\eta-\left(\frac{\eta}{2}-\frac{\alpha}{4}+\frac{1}{4}\right)}\Big([0,T];h^{2\left(\frac{\eta}{2}-\frac{\alpha}{4}+\frac{1}{4}\right)+\alpha}_L(\RR)\Big).
 \end{align*}
 With $\zeta=\eta+\frac{\alpha}{2}-\frac{1}{2}\in(0,1)$, we thus have (in particular) $
     \partial_sf\in\textup{BUC}^{\frac{\zeta}{2}}\big([0,T];h^{\zeta}_L(\RR)\big).$ 
 The continuity of $\partial_sf$ in $[0,T]\times[a,b]$ (since $\eta\geq\frac{1}{2}$) and the finiteness of the norm
 \begin{align*}
     \left\Vert \partial_sf\right\Vert_{H^{\frac{\zeta}{2},\zeta}\left(\overline{Q_T}\right)}
     =\,&\sup_{(t,s)\in{Q_T}}\left\vert\partial_s f\right\vert
     +\sup_{\substack{(t,s_1),(t,s_2)\in Q_T\\ s_1\neq s_2}}\frac{\left\vert \partial_sf(t,s_1)-\partial_sf(t,s_2)\right\vert}{\left\vert s_1-s_2\right\vert^{\zeta}}\\
     &+\sup_{\substack{(t_1,s),(t_2,s)\in Q_T \\ t_1\neq t_2}}\frac{\left\vert \partial_sf(t_1,s)-\partial_sf(t_2,s)\right\vert}{\left\vert t_1-t_2\right\vert^{\frac{\zeta}{2}}}
 \end{align*}
 imply that
 $
     \partial_sf\in H^{\frac{\zeta}{2},\zeta}\left([0,T]\times[a,b]\right).
 $
 \end{proof}

\section{Auxiliary results for Section \ref{sec:ste&dep}}
\label{sec:auxstep1}
In the following, we show Step 2 of the proof of Proposition \ref{prop:existence}, i.e.\ that $\textup{D}\Phi(\bar{u},\bar{\rho})$ (see \eqref{eq:Frechet2}) is a linear isomorphism. This is equivalent to show existence and uniqueness of a solution for a certain inhomogeneous linear initial value problem.
 
In this section, $\alpha \in (0,1)$ and $\eta \in \left(\frac12,1\right)$ are fixed.

\subsection{The linearization}

We compute here $\textup{D}\textup{\textbf{F}}(\bar{u},\bar{\rho})$ needed in \eqref{eq:Frechet2}, where $\textup{\textbf{F}}$ is given as in \eqref{F1F2}. By a direct calculation one finds
\begin{align}\nonumber
    \textup{D}\oF_1(\bar{u},\bar\rho)(v,\sigma)
    &=\beta(\bar\rho)\,\partial_s^2v+\beta^\prime(\bar\rho)\,\partial_s^2\bar{u}\;\sigma+\beta^\prime(\bar\rho)\,\partial_s\bar\rho\,\partial_sv\\
    &\quad+\beta^\prime(\bar\rho)\left(\partial_s\bar{u}+\partial_s\phi-c_0\right)\partial_s\sigma+\beta^{\prime\prime}(\bar\rho)\,\partial_s\bar\rho\left(\partial_s\bar{u}+\partial_s\phi-c_0\right)\sigma\\
    &\quad+\textup{D}\Lambda_u(\bar{u},\bar\rho)(v,\sigma), \label{DF1} \\ \nonumber
    \textup{D}\oF_2(\bar{u},\bar\rho)(v,\sigma)
    &=\mu\,\partial_s^2\sigma-\frac{1}{2}\beta^{\prime\prime}(\bar\rho)\left(\partial_s\bar{u}+\partial_s\phi-c_0\right)^2\sigma-\beta^\prime(\bar\rho)\left(\partial_s\bar{u}+\partial_s\phi-c_0\right)\partial_s v\\
    &\quad+\textup{D}\Lambda_\rho(\bar{u},\bar\rho)(v,\sigma).
\label{DF2}
\end{align}
Let $T>0$. We first observe that by \eqref{eq:LambdaupI}, in the Lagrange multipliers only  spatial derivatives of first order appear. More specifically,  $\Lambda_u$ is well defined with $\Lambda_u(u,\rho)\in \EE_0([0,T])$ for $(u, \rho)\in \mathbb{V}([0,T])$ where
\begin{align}\label{eq:E1/2} \mathbb{V}([0,T]):=\Big\{ &(u,\rho)\!\in \mathbb{F}([0,T])\!:=\!\textup{BUC}^1_{1-\eta}([0,T];h^\alpha_L(\RR))\cap \textup{BUC}_{1-\eta}([0,T];h^{1+\alpha}_L(\RR)) \nonumber\\
  &\text{such that } \inf_{[0,T]}\det\Pi(u+\phi)>0\Big\}.
\end{align}
This is an open subset of $\mathbb{F}([0,T])$ by continuity. Moreover, using $\eta>\frac12$ and \eqref{eq:tracespace}, any initial datum $(v_0, \sigma_0)\in \gamma\EE_1$ is in $\mathbb{F}([0,T])$ if we identify it with the function $(t, s)\mapsto (v_0(s), \sigma_0(s))$.
It is not difficult to see that
\begin{align}\label{eq:Lambdau mapping}
\mathbb{V}([0,T])    \ni (u, \rho)\mapsto\Lambda_u(u, \rho) \in \EE_0([0,T])
\end{align}
is of class $C^1$. Moreover, it follows from the product rule that
\begin{align}
    \textup{D}\Lambda_u(\bar{u}, \bar{\rho})(v,\sigma) &= \textup{D}\Pi^{-1}(\bar{u}+\phi)(v)\, \mathcal{J}(\bar{u}, \bar{\rho})\cdot\begin{pmatrix}
\sin(\bar{u}+\phi) \\ -\cos(\bar{u}+\phi)
    \end{pmatrix} \nonumber \\
    &\quad + \Pi^{-1}(\bar{u}+\phi)\, \textup{D}\mathcal{J}(\bar u, \bar{\rho})(v, \sigma) \cdot\begin{pmatrix}
\sin(\bar{u}+\phi) \\ -\cos(\bar{u}+\phi)
    \end{pmatrix}\nonumber\\
    &\quad +
    \Pi^{-1}(\bar{u}+\phi) \,\mathcal{J}(\bar{u}, \bar{\rho}) \cdot\begin{pmatrix}
\cos(\bar{u}+\phi) \\ \sin(\bar{u}+\phi)
    \end{pmatrix}\label{eq:DLambdau}
\end{align}
for all $(v, \sigma)\in \mathbb{F}([0,T])$. A similar argument yields that $\Lambda_\rho\colon \mathbb{V}([0,T])\to \EE_0([0,T])$ is well-defined and $C^1$ with derivative
\begin{align}
    \textup{D}\Lambda_\rho(\bar{u},\bar\rho)(v,\sigma)
    =\frac{1}{2L}\!\int_0^L\!\!\beta^{\prime\prime}(\bar\rho)\left(\partial_s\bar{u}+\partial_s\phi-c_0\right)^2\!\sigma+2\beta^\prime(\bar\rho)\left(\partial_s\bar{u}+\partial_s\phi-c_0\right)\partial_sv\intd s.
\end{align}

\subsection{Linear evolution problem with time-independent coefficients}

We consider an auxiliary linear evolution problem.
Since we rely on \cite{lC2011}, where time-\linebreak independent coefficients are used, we would like to freeze the coefficients of the terms in \eqref{DF1} and \eqref{DF2} in $t=0$. For the second term in \eqref{DF1}, the regularity of the initial datum does not allow this. Therefore, we treat this term differently from the others and move it together with $\textup{D}\Lambda_u$ and $\textup{D}\Lambda_\rho$, which are nonlocal, to the right hand side of the inhomogeneous initial value problem we want to solve.
This explains the structure of the differential operator chosen below.

\begin{prop}
\label{sollinprob}
Let $(u_0,\rho_0)\in\gamma\EE_1$ such that $u_0$ satisfies \eqref{eq:hypu}. 
Let $(\bar{u},\bar{\rho}) \in \EE_1([0,\bar{T}])$ be the function constructed in the proof of Proposition \ref{prop:existence}. Then for all $T\in (0, \bar{T}]$
the mapping 
\begin{align*}
    \oJ\colon&\EE_1([0,{T}])\to\EE_0([0,{T}])\times\gamma\EE_1,\quad
    (v,\sigma)\mapsto
    \big(
    \partial_t\left(v,\sigma\right)-\sA(v,\sigma) - \Psi,\;
    (v,\sigma)(0,\cdot)
    \big)
\end{align*}
where
\begin{align*}
    \sA(v,\sigma) :=&
    \Big( \beta(\rho_0)\partial_s^2v
    +\beta'(\rho_0)\partial_s\rho_0\partial_sv 
    +\beta'(\rho_0)(\partial_su_0+\partial_s\phi-c_0)\partial_s\sigma\\
    & \quad
    +\beta''(\rho_0)\partial_s\rho_0(\partial_su_0+\partial_s\phi-c_0)\sigma,\\
    &\quad \mu\partial_s^2\sigma-\frac{1}{2}\beta''(\rho_0)(\partial_su_0+\partial_s\phi-c_0)^2\sigma-\beta'(\rho_0)(\partial_su_0+\partial_s\phi-c_0)\partial_sv \Big), \\
    \Psi=&\;\Psi(v_0,\sigma_0) := 
    \Big( \beta'(\bar\rho)\partial_s^2\bar{u}\,\sigma_0
   +\textup{D}\Lambda_u(\bar{u},\bar{\rho})(v_0,\sigma_0),\;\textup{D}\Lambda_\rho(\bar{u},\bar\rho)(v_0,\sigma_0)\Big)
\end{align*}
is well-defined and an isomorphism. Here $(v_0,\sigma_0)=(v,\sigma)(0,\cdot)\in \mathbb{F}([0,{T}])$, cf.\ \eqref{eq:E1/2}.
Equivalently\footnote{Here we slightly abuse notation, denoting by $(v_0, \sigma_0)$ both the initial datum of the problem \eqref{eq:probD1} and the evaluation at $t=0$ of $(v, \sigma)\in \EE_1([0,{T}])$. However, for solving \eqref{eq:probD1}, this does not make a difference.}, for any pair
$(\varphi_1,\varphi_2)\in\EE_0([0,{T}])$ and $(v_0,\sigma_0) \in \gamma \EE_1$ there exists a unique solution
$(u,\rho)\in\EE_1([0,{T}])$
of the linear problem 
\begin{align}
\begin{split}
    \label{eq:probD1}
    \begin{cases}
    \partial_t(v,\sigma)-\sA(v,\sigma) =(\varphi_1,\varphi_2)+ \Psi &\textup{ in } (0,{T})\times\RR,\\
    (v,\sigma)(0,\cdot)=(v_0,\sigma_0) &\textup{ on } \RR.
    \end{cases}
\end{split}    
\end{align}
\end{prop}

\begin{proof}
We want to use \cite{lC2011}, which works with functions on the one-dimensional torus, or equivalently for periodic functions, cf.\ \cite[Proposition 1.1]{lC2011}. 
With \eqref{eq:tracespace} and Lemma \ref{multlith},  the time-independent coefficients of the operator $\sA$ are in $h^{\alpha}_L(\RR)$. Because $\beta(\rho_0)$ is bounded away from zero and $\mu>0$, $\sA$ is a uniformly elliptic operator.
Apart from this we have to check that $\Psi 
\in\EE_0([0,{T}])$. 
Using Lemma \ref{multlith} and Remark \ref{rem:appbuc2}, we find $\beta^\prime(\bar{\rho})\in \EE_1([0,{T}])$ and thus $\beta^\prime(\bar\rho)\partial_s^2\bar{u}\,\sigma_0\in\EE_0([0,{T}])$.
From \eqref{eq:Lambdau mapping}, we conclude that $\textup{D}\Lambda_u(\bar{u},\bar{\rho})(v_0, \sigma_0)\in \EE_0([0,{T}])$ and similarly for $\Lambda_\rho$.
Applying \cite[Theorem 6.4]{lC2011} yields a unique solution $(v,\sigma)\in\EE_1([0,{T}])$ of  \eqref{eq:probD1}.
\end{proof}

From the fact that the mapping $\oJ$ from Proposition \ref{sollinprob} is an isomorphism, an estimate of the norm of the solution follows implicitly. We now show that the constant in this estimate can be chosen independent of $T$.

\begin{lem} 
\label{normabsch}
Let $(u_0,\rho_0)\in\gamma\EE_1$ such that $u_0$ satisfies \eqref{eq:hypu}.  Let $(\bar{u},\bar\rho)\in\EE_1([0,\bar{T}])$ be the map constructed in the proof of Proposition \ref{prop:existence}. Let $T\in (0, \bar{T}]$. Then the solution operator $\oJ^{-1}$ satisfies the estimate
\begin{align*}
     \left\Vert\oJ^{-1}  \big((\varphi_1,\varphi_2),(v_0,\sigma_0)\big)\right\Vert_{\EE_1([0,T])}\leq C \left(\left\Vert \left(\varphi_1,\varphi_2\right)\right\Vert_{\EE_0([0,T])}+\left\Vert(v_0,\sigma_0)\right\Vert_{\gamma\EE_1}\right) 
 \end{align*}
for all $(\varphi_1,\varphi_2)\in\EE_0([0,T])$ and $(v_0,\sigma_0) \in \gamma\EE_1$ with $C=C(u_0,\rho_0,\bar{T})$.
\end{lem}
Notice that here we do not write explicitely the dependence on $T$ of the solution operator. This abuse of notation is justified in the proof below. 
\begin{proof}[Proof of Lemma \ref{normabsch}]
Proposition \ref{sollinprob} already yields the claimed inequality with $C=C(u_0,\rho_0,T)$. We show that $C$ can be chosen only depending on $\bar{T}>T$ by an extension argument. Let $(v,\sigma)=\oJ^{-1} \big((\varphi_1,\varphi_2),(v_0,\sigma_0)\big)\in \EE_1([0,T])$. Consider
\begin{align*}
    (\varphi_{1}^*,\varphi_{2}^*)(t,s)=
    \begin{cases}
    (\varphi_{1},\varphi_{2})(t,s) &\textup{ for } (t,s)\in[0,T]\times\RR,\\
    (\varphi_{1},\varphi_{2})(T,s) &\textup{ for } (t,s)\in(T,\bar{T}]\times\RR.
    \end{cases}
\end{align*}
Then $(\varphi_{1}^*,\varphi_{2}^*)\in\EE_0([0,\bar{T}])$ and we can apply Proposition \ref{sollinprob} and obtain 
$(v^*,\sigma^*)=\oJ^{-1} \big((\varphi_1^*,\varphi_2^*),(v_0,\sigma_0)\big)\in\EE_1([0,\bar{T}])$, the  solution of \eqref{eq:probD1} on $(0,\bar{T}]$ with $(\varphi_1^*, \varphi_2^*)$ instead of $(\varphi_1, \varphi_2)$. The same proposition gives that $\oJ^{-1} $ is an isomorphism from $\EE_1([0,\bar{T}])$ to $\EE_0([0,\bar{T}])$. Hence, there exists a constant $C=C(u_0,\rho_0,\bar{T})$ such that
\begin{align*}
    \left\Vert(v^*,\sigma^*)\right\Vert_{\EE_1([0,\bar{T}])}\leq C \left(\left\Vert \left(\varphi_1^*,\varphi_2^*\right)\right\Vert_{\EE_0([0,\bar{T}])}+\left\Vert(v_0,\sigma_0)\right\Vert_{\gamma\EE_1}\right).
\end{align*}
Because of the uniqueness of the solution and using that \eqref{eq:probD1} is local in time, we find $(v^*,\sigma^*)\vert_{[0,T]}=(v,\sigma)$. The claim follows using that  $
\left\Vert(v,\sigma)\right\Vert_{\EE_1([0,T])}\leq\left\Vert(v^*,\sigma^*)\right\Vert_{\EE_1([0,\bar{T}])},$
together with
$\left\Vert\left(\varphi_1^*,\varphi_2^*\right)\right\Vert_{\EE_0([0,\bar{T}])}=\left\Vert\left(\varphi_1,\varphi_2\right)\right\Vert_{\EE_0([0,T])}$.
\end{proof}

\subsection{Some auxiliary estimates}
\label{sec:auxest}
We collect here some technical results, which we use to examine the differences appearing when comparing the linear operators in \eqref{eq:Frechet2} and $\mathbf{J}$ from Proposition \ref{sollinprob}.

Throughout this subsection, we fix $(u_0, \rho_0)\in \gamma\EE_1$ such that $u_0$ satisfies \eqref{eq:hypu} and denote by $(\bar{u}, \bar{\rho})\in \EE_1([0,\bar{T}])$ the function constructed in the proof of Proposition \ref{prop:existence}. In particular, $\eta\in (\frac{1}{2},1)$, $\tilde{\eta}=\eta-\frac{1}{2}$ and $\partial_s (\bar{u}, \bar{\rho})\in \tilde{\EE}_1([0,\bar{T}])$.
\begin{lem}\label{est:v-v0}
Let $t\in (0,\bar{T}]$. Then we have
\begin{align}
\left\Vert \bar{u}(t)-u_0\right\Vert_{C^\alpha}\leq \left\Vert \bar{u}\right\Vert_{\EE_1}\frac{t^\eta}{\eta}, \qquad \left\Vert \partial_s\bar{u}(t)-\partial_s u_0\right\Vert_{C^\alpha}&\leq \left\Vert \partial_s\bar{u}\right\Vert_{\tilde{\EE}_1}\frac{t^{\tilde{\eta}}}{\tilde{\eta}},
\end{align}
and similarly for $\bar{\rho}$.
\end{lem}
\begin{proof}
Inspired by \eqref{eq:martin1}, we compute
\begin{align}\label{estrho0-rho}
    &\big\Vert \bar{u}(t)-u_0 \big\Vert_{C^{\alpha}}
    =\big\Vert\lim_{t'\to0+} (\bar{u}(t)-\bar{u}(t'))\big\Vert_{C^{\alpha}}
    =\big\Vert\lim_{t'\to0+}\int_{t'}^t\partial_\tau \bar{u} (\tau)\intd \tau\big\Vert_{C^{\alpha}}
    \\
    &\quad \leq\lim_{t'\to0+}\int_{t'}^t\left\Vert\tau^{\eta-1}\tau^{1-\eta}\partial_\tau \bar{u} (\tau)\right\Vert_{C^\alpha}\intd \tau\leq \left\Vert \bar{u} \right\Vert_{\EE_1}\lim_{t'\to0+}\int_{t'}^t\tau^{\eta-1}\intd \tau
    =\left\Vert \bar{u} \right\Vert_{\EE_1}\frac{t^\eta}{\eta}.\nonumber
\end{align}
For $\partial_s \bar{u}$, we proceed completely analogously.
\end{proof}

In the sequel, we need to estimate terms involving $\beta^{(i)}(\bar{\rho})$. Since we have $\bar{\rho}\in \textup{BUC}([0,\bar{T}]; h_L^{2\eta+\alpha}(\R))$, see Lemma \ref{appbuc1}, there exists $J\subset \R$, a compact interval, such that $\bar{\rho}(t,s)\in J$ for all $(t,s)\in [0,\bar{T}]\times \R$. Since $\beta\in C^\infty(\RR)$, there is a constant $M\in (0,\infty)$ such that
\begin{align}\label{eq:estimatereferencesol}
\begin{split}
    &\sup_{t \in [0,\bar{T}]}\left\Vert \Pi^{-1}(\bar{u}+\phi)\right\Vert\leq M,\quad
    \sup_{J}\big\vert\beta^{(i)}\big\vert\leq M \;\text{ for } i=0, \dots, 4,
    \\
    &\sup_{t \in [0,\bar{T}]} \| \bar{u}\|_{C^{1}},\;  
    \| \bar{u}\|_{\EE_1([0,\bar{T}])},\; 
    \sup_{t \in [0,\bar{T}]}\| \bar{\rho}\|_{C^1}\leq M.
  \end{split}
\end{align}

\begin{lem}
\label{estbetarho}
Let $M>0$ be as in \eqref{eq:estimatereferencesol}. There is $C=C(M)$ such that for all $i\in\{0, 1, 2\}$ we have
\begin{align*}
    \Vert\beta^{(i)}(\rho_0)-\beta^{(i)}(\bar\rho(t))\Vert_{C^{\alpha}}
    \leq
    C(M)\left\Vert\bar\rho\right\Vert_{\EE_1}\frac{t^\eta}{\eta} \quad\text{ for all } t \in[0,\bar{T}].
\end{align*}
\end{lem}
\begin{proof}
For the supremum norm we find 
\begin{align}
\label{estsupnorm}
    \left\Vert \beta(\rho_0)-\beta({\bar\rho})\right\Vert_{C^0}
    &=\sup_{s\in\RR}\left\vert\beta(\rho_0)-\beta({\bar\rho})\right\vert
    \leq\sup_{J}\left\vert\beta^\prime\right\vert\sup_{s\in\RR}\left\vert\rho_0-{\bar\rho}\right\vert.
\end{align}
To estimate the Hölder semi-norm, we note that for $s_1,s_2\in\RR$
\begin{align*}
    &\left\vert\beta(\rho_0(s_1))-\beta(\bar\rho(s_1))-\beta(\rho_0(s_2))+\beta(\bar\rho(s_2))\right\vert\\
    &=\left\vert\int_0^1\partial_\lambda\left[\beta\left(\lambda\rho_0(s_1)+(1-\lambda)\bar\rho(s_1)\right)-\beta\left(\lambda\rho_0(s_2)+(1-\lambda){\bar\rho}(s_2)\right)\right]\intd\lambda\right\vert\\
    &\leq\int_0^1\!\left\vert\beta^\prime\left(\lambda\rho_0(s_1)+(1-\lambda)\bar\rho(s_1)\right)\right\vert\left\vert\rho_0(s_1)-\bar\rho(s_1)-\rho_0(s_2)+\bar\rho(s_2)\right\vert\intd\lambda\\
    &+\int_0^1\!\left\vert\beta^\prime\left(\lambda\rho_0(s_1)+(1-\lambda)\bar\rho(s_1)\right)-\beta^\prime\left(\lambda\rho_0(s_2)+(1-\lambda)\bar\rho(s_2)\right)\right\vert\left\vert\rho_0(s_2)-\bar\rho(s_2)\right\vert\intd\lambda\\
    &\leq\sup_{J}\left\vert\beta^\prime\right\vert\,\left[\rho_0-\bar\rho\right]_{\alpha}\left\vert s_1-s_2\right\vert^{{\alpha}}\\
    &\quad+\sup_{J}\left\vert\beta^{\prime\prime}\right\vert\left(\left\vert\rho_0(s_1)-\rho_0(s_2)\right\vert+\left\vert\bar\rho(s_1)-\bar\rho(s_2)\right\vert\right)\left\vert\rho_0(s_2)-\bar\rho(s_2)\right\vert\\
    &\leq \Big[ \sup_{J}\left\vert\beta^\prime\right\vert\,\left[\rho_0-\bar\rho\right]_{\alpha}+\sup_{J}\left\vert\beta^{\prime\prime}\right\vert\left(\left[\rho_0\right]_\alpha+\left[\bar\rho\right]_\alpha\right)\sup_{ s\in \R}\left\vert\rho_0-\bar\rho\right\vert \Big]\left\vert s_1-s_2\right\vert^{{\alpha}}.
\end{align*}
This estimate together with \eqref{estsupnorm}
and Lemma \ref{est:v-v0} yields the claim for $i=0$. Since $\beta$ is smooth, we can proceed in the same way for $i\geq 1$.
\end{proof}

\begin{lem}
\label{est:DLambda}
Let $M>0$ be as in \eqref{eq:estimatereferencesol}. Then there exists a constant $C$, only depending on $M$, $\eta$, $\alpha$ and the model parameters, such that for all $T\in (0, \bar{T}]$, $(v,\sigma)\in\EE_{1}([0,T])$ and $(v_0, \sigma_0):=(v,\sigma)(0,\cdot)\in\gamma\EE_1$ we have
\begin{align}\label{eq:estlambdau}
    \qquad\left\Vert\textup{D}\Lambda_u(\bar{u},\bar\rho)(v,\sigma)-\textup{D}\Lambda_u(\bar{u},\bar\rho)(v_0,\sigma_0)\right\Vert_{\EE_0}
    &\leq{C}\left(\left\Vert v\right\Vert_{\EE_1}+\left\Vert\sigma\right\Vert_{\EE_1}\right)\max\left\lbrace T,T^\eta\right\rbrace,\\
\label{eq:estlambdarho}
    \left\Vert\textup{D}\Lambda_\rho(\bar{u},\bar\rho)(v,\sigma)-\textup{D}\Lambda_\rho(\bar{u},\bar\rho)(v_0,\sigma_0)\right\Vert_{\EE_0}
    &\leq{C}\left(\left\Vert v\right\Vert_{\EE_1}+\left\Vert\sigma\right\Vert_{\EE_1}\right)\max\left\lbrace T,T^\eta\right\rbrace.
\end{align}
\end{lem}
Here $\textup{D}\Lambda_u$ is the Fréchet derivative of the map defined in \eqref{eq:Lambdau mapping}, thus may be applied to $(v_0, \sigma_0)$, identified with an element of $\mathbb{F}([0,T])$, cf.\ \eqref{eq:E1/2}. Similarly for $\Lambda_\rho$.

\begin{proof}[Proof of Lemma \ref{est:DLambda}]
In the following, $C$ denotes a constant which varies from line to line but only depends on $M$ and the model parameters.
First, we prove \eqref{eq:estlambdau} considering each term in  \eqref{eq:DLambdau} separately.
By \eqref{eq:def Pi} with $\theta=\bar{u}+\phi$, we find 
\begin{align}
\label{eq:DPi}
\qquad
    \textup{D}\Pi\left(\bar{u}+\phi\right)(v)=\begin{pmatrix}
    \int_0^L\sin\left(2\left(\bar{u}+\phi\right)\right)v\intd s 
    & - \int_0^L\cos\left(2\left(\bar{u}+\phi\right)\right)v\intd s \\
    - \int_0^L\cos\left(2\left(\bar{u}+\phi\right)\right)v\intd s
    & -\int_0^L\sin\left(2\left(\bar{u}+\phi\right)\right)v\intd s
    \end{pmatrix}
\end{align}
and obtain using \eqref{eq:estimatereferencesol} and \eqref{eq:DPi} (considering the operator norm of the matrix)
\begin{align*}
    &\left\Vert\textup{D}\Pi^{-1}\left(\bar{u}+\phi\right)(v)-\textup{D}\Pi^{-1}\left(\bar{u}+\phi\right)(v_0)\right\Vert\\
    &\qquad= \left\Vert-\Pi^{-1}\left(\bar{u}+\phi\right)\big(\textup{D}\Pi\left(\bar{u}+\phi\right)(v)-\textup{D}\Pi\left(\bar{u}+\phi\right)(v_0)\big)\,\Pi^{-1}\left(\bar{u}+\phi\right)\right\Vert\\
    &\qquad
    \leq M^2 \sqrt{2}L\sup_{s\in[0,L]}\left\vert v-v_0 \right\vert
    \leq C\left\Vert v-v_0 \right\Vert_{C^\alpha}.
\end{align*}
Hence, with Lemma  \ref{est:v-v0} the difference coming from the first term on the right hand side of \eqref{eq:DLambdau} can be estimated by
\begin{align*}
    &\Big\Vert
    \big(\textup{D}\Pi^{-1}\left(\bar{u}+\phi\right)(v)-\textup{D}\Pi^{-1}\left(\bar{u}+\phi\right)(v_0)\big)
    \,\mathcal{J}(\bar{u}, \bar\rho)
    \cdot
    \begin{pmatrix}
    \sin\left(\bar{u}+\phi\right) \\ -\cos\left(\bar{u}+\phi\right)
    \end{pmatrix}
    \Big\Vert_{\EE_0}\\
    &\qquad \leq C\sup_{t\in(0,T]}t^{1-\eta}
    \left\Vert v-v_0 \right\Vert_{C^\alpha}
    \leq C\left\Vert v \right\Vert_{\EE_1}T.
\end{align*}

For the second term of \eqref{eq:DLambdau} we integrate once by parts and obtain
\allowdisplaybreaks{\begin{align}
\label{eq:secondD}
    \textup{D}\mathcal{J}(\bar{u}, \bar{\rho})(v,\sigma) \nonumber&=
\int_0^L\begin{pmatrix}
-\sin\left(\bar{u}+\phi\right) \\ \cos\left(\bar{u}+\phi\right)
\end{pmatrix}v
\left(\partial_s\bar{u}+\partial_s\phi\right)\beta(\bar{\rho})\left(\partial_s\bar{u}+\partial_s\phi-c_0\right)\intd s \nonumber\\
    &\qquad+\int_0^L\begin{pmatrix}
\cos\left(\bar{u}+\phi\right) \\ \sin\left(\bar{u}+\phi\right)
\end{pmatrix}
\left(\partial_s\bar{u}+\partial_s\phi\right)\beta^\prime(\bar{\rho})\sigma\left(\partial_s\bar{u}+\partial_s\phi-c_0\right)\intd s \nonumber\\
    &\qquad-\int_0^L\begin{pmatrix}
-\sin\left(\bar{u}+\phi\right) \\ \cos\left(\bar{u}+\phi\right)
\end{pmatrix}
\left(\partial_s\bar{u}+\partial_s\phi\right)\left(2\partial_s\bar{u}+2\partial_s\phi-c_0\right)\beta(\bar{\rho})\,v\intd s \nonumber\\
    &\qquad-\int_0^L \begin{pmatrix}
\cos\left(\bar{u}+\phi\right) \\ \sin\left(\bar{u}+\phi\right)
\end{pmatrix} 2\partial_s^2\bar{u}\,\beta(\bar{\rho})\,v\intd s \nonumber\\
    &\qquad-\int_0^L \begin{pmatrix}
\cos\left(\bar{u}+\phi\right) \\ \sin\left(\bar{u}+\phi\right)
\end{pmatrix} \left(2\partial_s\bar{u}+2\partial_s\phi-c_0\right)\beta'(\bar{\rho})\partial_s\bar{\rho} \,v\intd s.
\end{align}
}So the difference we need to consider can be estimated by
\begin{align}
\label{eq:thirdD}
    &\big\Vert \textup{D}\mathcal{J}(\bar{u}, \bar{\rho})(v_0,\sigma_0) -\textup{D}\mathcal{J}(\bar{u}, \bar{\rho})(v,\sigma)\big\Vert_{C^0} \nonumber\\
    &\quad\leq 
     C\left(1+\left\Vert\partial_s^2\bar{u}\right\Vert_{C^\alpha}\right) \left\Vert v-v_0\right\Vert_{C^\alpha} + C\left\Vert \sigma-\sigma_0\right\Vert_{C^\alpha} \nonumber\\
    & \quad\leq C t^{\eta} \left(1+\left\Vert\partial_s^2\bar{u}\right\Vert_{C^\alpha}\right)\left(\left\Vert v\right\Vert_{\EE_1}+\left\Vert\sigma\right\Vert_{\EE_1}\right).
\end{align}
In the last step we used Lemma \ref{est:v-v0}.
Using \eqref{eq:estimatereferencesol}, we obtain 
\begin{align*}
    &\Big\Vert
    \Pi^{-1}\left(\bar{u}+\phi\right) \Big(\textup{D}\mathcal{J}(\bar{u}, \bar{\rho})(v,\sigma) - \textup{D}\mathcal{J}(\bar{u}, \bar{\rho})(v_0,\sigma_0)\Big)
\cdot\begin{pmatrix}
\sin\left(\bar{u}+\phi\right) \\ -\cos\left(\bar{u}+\phi\right)
\end{pmatrix}\Big\Vert_{\EE_0}\\
&\leq C \sup_{t\in(0,T]}t^{1-\eta} \; t^{\eta} (1+\left\Vert\partial_s^2\bar{u}\right\Vert_{C^\alpha})\left(\left\Vert v\right\Vert_{\EE_1}+\left\Vert\sigma\right\Vert_{\EE_1}\right) \leq C \left(\left\Vert v\right\Vert_{\EE_1}+\left\Vert\sigma\right\Vert_{\EE_1}\right) (T+T^\eta).
\end{align*}
Finally, we look at the difference coming from the third term of \eqref{eq:DLambdau}. We have
\begin{align*}
    \left\Vert\begin{pmatrix}
        \cos\left(\bar{u}+\phi\right) \\ \sin\left(\bar{u}+\phi\right)
        \end{pmatrix}
        (v-v_0)\right\Vert_{C^{\alpha}} 
        \leq C \left\Vert v-v_0\right\Vert_{C^\alpha}.
\end{align*}
Thereby we obtain using \eqref{eq:estimatereferencesol} and Lemma \ref{est:v-v0}
\begin{align*}
    &\bigg\Vert\Pi^{-1}\left(\bar{u}+\phi\right)\,\mathcal{J}(\bar{u}, \bar{\rho})
\cdot\begin{pmatrix}
        \cos\left(\bar{u}+\phi\right) \\ \sin\left(\bar{u}+\phi\right)
        \end{pmatrix}(v-v_0)\bigg\Vert_{\EE_0}\\
        &\qquad\leq C\sup_{t\in(0,T]}t^{1-\eta}\left\Vert v-v_0\right\Vert_{C^\alpha}\leq C\left\Vert v\right\Vert_{\EE_1}T.
\end{align*}
This proves the first part of the claim.
The second part follows similarly. 
\end{proof}

\subsection{Solving the linearized problem}
\label{sec:solofstep1}
We show that the map in \eqref{eq:Frechet2} is an isomorphism, now considering
a linear initial value problem with time-dependent coefficients.

\begin{prop}
\label{prop:sollinprobt}
Let $(u_0,\rho_0)\in\gamma\EE_1$ such that $u_0$ satisfies \eqref{eq:hypu} and let $(\bar{u},\bar{\rho})\in \EE_1([0,\bar{T}])$ be the function constructed in the proof of Proposition \ref{prop:existence}.
Then there exists $T^\prime\in(0,\bar{T}]$ such that the mapping 
\begin{align*}
    \tilde{\oJ}\colon \EE_1([0,T^\prime])\to\EE_0([0,T^\prime])\times\gamma\EE_1, \; (v,\sigma)\mapsto
    \big(
    \partial_t(v,\sigma)-\textup{D}\oF(\bar{u},\bar{\rho})(v,\sigma),\,
    (v,\sigma)(0,\cdot)
    \big)
\end{align*}
is an isomorphism. Equivalently, for any right-hand side $(\varphi_1,\varphi_2)\in\EE_0([0,T'])$ and any initial datum $(v_0,\sigma_0)\in\gamma\EE_1$ there is a unique solution $(v,\sigma)\in\EE_1([0,T^\prime])$ of
\begin{align}
\label{eq:evprob lin}
    \begin{cases}
    \partial_t(v,\sigma)-\textup{D}\oF(\bar{u},\bar{\rho})(v,\sigma)=(\varphi_1,\varphi_2) & \mbox{ in }(0,T') \times \R,\\
    (v,\sigma)(0,\cdot)=(v_0,\sigma_0)  & \mbox{ on } \R.
    \end{cases}
\end{align}
\end{prop}

\begin{proof}
Let $T\in (0, \bar{T}]$ to be chosen.
We compare $\tilde{\oJ}  $ with $\oJ $ from Proposition \ref{sollinprob}, which we already know is an isomorphism. We have
\begin{align*}
    \tilde{\oJ}  (v,\sigma)
    &=\oJ(v,\sigma)
    +\big(
   \sA(v,\sigma)+\Psi(v_0, \sigma_0)-\textup{D}\oF(\bar{u},\bar{\rho})(v,\sigma),\,
    (0,0)
    \big)\\
    &=\oJ\Big(\textup{\textbf{I}}(v,\sigma)+\oJ^{-1} (v,\sigma)\,
    \big(
    \textup{\textbf{S}}(v,\sigma),\, (0,0)
    \big)\Big),
\end{align*}
where $\sA ,\Psi$ are given as in Proposition \ref{sollinprob}, $\textup{\textbf{I}}$ is the operator identity and 
\begin{align}
  &\textbf{S}:=(\textbf{S}_1, \textbf{S}_2)\colon \EE_1([0,T]) \to \EE_0([0,T]),\\
  &\textbf{S}(v,\sigma)= \sA(v,\sigma)+\Psi(v_0, \sigma_0)-\textup{D}\oF(\bar{u},\bar{\rho})(v,\sigma).  
\end{align}

Thus, by a Neumann series argument, it is sufficient to show that the operator norm of $\textbf{S}$  
becomes arbitrarily small for $T$ getting small and at the same time the operator norm $\left\Vert\oJ^{-1}   \right\Vert$ remains bounded, independent of $T\leq \bar{T}$. This last part has already been proven in Lemma \ref{normabsch}. Hence, the statement follows if we show
\begin{align}\label{eq:pfannkuchen}
    \left\Vert \textup{\textbf{S}}\right\Vert\leq C \max\left\lbrace T, T^\eta,T^{\tilde\eta}\right\rbrace.
\end{align}

By definition of  $\sA ,\Psi$ and using  
\eqref{DF1}, \eqref{DF2}
we find
\begin{align}
\label{eq:S1}
    \textup{\textbf{S}}_1(v,\sigma)
    &=\left(\beta(\rho_0)-\beta(\bar\rho)\right)\partial_s^2v 
    +\beta^\prime(\bar\rho)\,\partial_s^2\bar{u}\left(\sigma_0-\sigma\right) 
    +\left(\beta^\prime(\rho_0)\partial_s\rho_0-\beta^\prime(\bar\rho)\partial_s\bar\rho\right)\partial_sv \nonumber\\
    &\quad
    +\left(\beta'(\rho_0)\left(\partial_su_0+\partial_s\phi-c_0\right)-\beta'(\bar\rho)\left(\partial_s\bar{u}+\partial_s\phi-c_0\right)\right)\partial_s\sigma \nonumber\\
    &\quad
    +\left(\beta''(\rho_0)\partial_s\rho_0\left(\partial_su_0+\partial_s\phi-c_0\right)-\beta''(\bar\rho)\partial_s\bar\rho\left(\partial_s\bar{u}+\partial_s\phi-c_0\right)\right)\sigma \nonumber\\
    &\quad
    +\textup{D}\Lambda_u(\bar{u},\bar\rho)(v_0,\sigma_0)-\textup{D}\Lambda_u(\bar{u},\bar\rho)(v,\sigma)
\end{align}
and
\begin{align}
\label{eq:S2}
    \textup{\textbf{S}}_2(v,\sigma)
    &=\Big(-\frac{1}{2}\beta^{\prime\prime}(\rho_0)\left(\partial_su_0+\partial_s\phi-c_0\right)^2+\frac{1}{2}\beta^{\prime\prime}(\bar\rho)\left(\partial_s\bar{u}+\partial_s\phi-c_0\right)^2\Big)\sigma \nonumber\\
    &\quad
    +\left(-\beta^\prime(\rho_0)\left(\partial_su_0+\partial_s\phi-c_0\right) +\beta^\prime(\bar\rho)\left(\partial_s\bar{u}+\partial_s\phi-c_0\right)\right)\partial_s v \nonumber\\
    &\quad
    +\textup{D}\Lambda_\rho(\bar{u},\bar\rho)(v_0,\sigma_0)-\textup{D}\Lambda_\rho(\bar{u},\bar\rho)(v,\sigma).
\end{align}
We now take a look at the terms of \eqref{eq:S1} and \eqref{eq:S2}.
In the following, $C$ denotes a constant, only depending on the model parameters, $\alpha$, $\eta$ and $M$ as in \eqref{eq:estimatereferencesol}, which is allowed to change from line to line.
The $\EE_0$-norm of the first term of \eqref{eq:S1} can be estimated using Lemma \ref{multlith} and Lemma \ref{estbetarho} by
\begin{align*}
    \left\Vert \left(\beta(\rho_0)-\beta(\bar\rho)\right)\partial_s^2v \right\Vert_{\EE_0}
    &\leq C \sup_{t\in(0,T]}t^{1-\eta}\left\Vert\beta(\rho_0)-\beta(\bar\rho)\right\Vert_{C^\alpha}\left\Vert v\right\Vert_{C^{2+\alpha}}
    \leq C\left\Vert v\right\Vert_{\EE_1}T^\eta.
\end{align*}
Here we used the time-weight of the $\EE_0$-norm to bound the second derivative of $v$. For the second term of \eqref{eq:S1} we have to use the time-weight to bound the second derivative of $\bar{u}$. An additional factor depending on $T$ and going to zero for $T$ going to zero is given by Lemma \ref{est:v-v0}. This motivates the choice of the first term in $\Psi$ in Proposition \ref{sollinprob}. We obtain
\begin{align*}
    \left\Vert\beta^\prime(\bar\rho)\,\partial_s^2\bar{u}\left(\sigma_0-\sigma\right)\right\Vert_{\EE_0}
    &\leq C\sup_{t\in(0,T]}t^{1-\eta}\left\Vert\beta^\prime(\bar\rho)\right\Vert_{C^\alpha}\left\Vert \bar{u}\right\Vert_{C^{2+\alpha}}\left\Vert\sigma_0-\sigma\right\Vert_{C^\alpha}\\
    &\leq C\left\Vert\sigma\right\Vert_{\EE_1}T^\eta.
\end{align*}
The third term of \eqref{eq:S1} can be estimated using Lemma \ref{estbetarho} and Lemma \ref{est:v-v0} by
\begin{align*}
&\left\Vert \left(\beta^\prime(\rho_0)\partial_s\rho_0-\beta^\prime(\bar\rho)\partial_s\bar\rho\right)\partial_sv\right\Vert_{\EE_0}\\
&\leq C\sup_{t\in(0,T]}t^{1-\eta}\left\Vert\beta'(\rho_0)-\beta'(\bar\rho)\right\Vert_{C^\alpha}\left\Vert\rho_0\right\Vert_{C^{1+\alpha}}\left\Vert v\right\Vert_{C^{1+\alpha}}\\
&\quad+C\sup_{t\in(0,T]}t^{1-\eta}\left\Vert\beta'(\bar\rho)\right\Vert_{C^\alpha}\left\Vert\partial_s\rho_0-\partial_s\bar\rho\right\Vert_{C^{\alpha}}\left\Vert v\right\Vert_{C^{1+\alpha}}\leq C\left\Vert v\right\Vert_{\EE_1}T^\eta+C\left\Vert v\right\Vert_{\EE_1}T^{\tilde\eta}.
\end{align*}
The same also applies for the fourth term of \eqref{eq:S1} and the second term of \eqref{eq:S2}.
For the fifth term of \eqref{eq:S1} we proceed similarly and estimate
\begin{align*}
    &\left\Vert\left(\beta''(\rho_0)\partial_s\rho_0\left(\partial_su_0+\partial_s\phi-c_0\right)-\beta''(\bar\rho)\partial_s\bar\rho\left(\partial_s\bar{u}+\partial_s\phi-c_0\right)\right)\sigma\right\Vert_{\EE_0}\\
    &\leq C\!\sup_{t\in(0,T]}t^{1-\eta}\left\Vert\beta''(\rho_0)-\beta''(\bar\rho)\right\Vert_{C^\alpha}\left\Vert\rho_0\right\Vert_{C^{1+\alpha}}\left\Vert\partial_su_0+\partial_s\phi-c_0\right\Vert_{C^\alpha}\left\Vert\sigma\right\Vert_{C^\alpha}\\
    &\quad+C\!\sup_{t\in(0,T]}t^{1-\eta}\left\Vert\beta''(\bar\rho)\right\Vert_{C^\alpha}\left\Vert\partial_s\rho_0-\partial_s\bar\rho\right\Vert_{C^\alpha}\left\Vert\partial_su_0+\partial_s\phi-c_0\right\Vert_{C^\alpha}\left\Vert\sigma\right\Vert_{C^\alpha}\\
    &\quad+C\!\sup_{t\in(0,T]}t^{1-\eta}\left\Vert\beta''(\bar\rho)\right\Vert_{C^\alpha}\left\Vert\bar\rho\right\Vert_{C^{1+\alpha}}\left\Vert\partial_su_0-\partial_s\bar{u}\right\Vert_{C^\alpha}\left\Vert\sigma\right\Vert_{C^\alpha}\leq C\left\Vert\sigma\right\Vert_{\EE_1}\!\left(T^\eta+T^{\tilde\eta}\right).
\end{align*}
Analogously, we can treat the first term of \eqref{eq:S2}. 
Together with Lemma \ref{est:DLambda}, \eqref{eq:pfannkuchen} and thus the claim follows for $T'=T<\bar{T}$ sufficiently small.
\end{proof}

\section*{Acknowledgements} The authors acknowledge support by the DFG (German Research Foundation), project no.\ 404870139. The third author is additionally supported by the Austrian Science Fund (FWF) project/grant P 32788-N.

\bibliographystyle{abbrv}
\bibliography{biblio}

\end{document}